\documentclass{siam}
\usepackage[T1]{fontenc}
\usepackage{amsmath,amssymb}
\usepackage{graphicx}
\numberwithin{theorem}{section}

\newcommand{\TheTitle}{Optimality of a standard adaptive finite element method for the Stokes problem} 
\newcommand{\TheAuthors}{M.~Feischl}

\headers{Optimal adaptivity for the Stokes problem}{\TheAuthors}

\title{{\TheTitle}\thanks{Submitted to the editors DATE.
		\funding{Supported by the Australian Research Council (ARC) under
			grant number DE170100222 and by the Deutsche Forschungsgemeinschaft (DFG) through CRC 1173. }}}

\author{
	Michael Feischl\thanks{Institute for Numerical Simulation, Universit\"at Bonn,
Endenicher Allee 19b,
53115 Bonn
		(\email{michael.feischl@uni-bonn.de}).}
}

\usepackage{amsopn}
\def\A{\mathbb A}

\def\R{{\mathbb R}}
\def\N{{\mathbb N}}

\def\EE{{\mathcal E}}
\def\BB{{\mathcal B}}

\def\MM{{\mathcal M}}
\def\NN{{\mathcal N}}
\def\OO{{\mathcal O}}
\def\PP{{\mathcal P}}
\def\RR{{\mathcal R}}
\def\SS{{\mathcal S}}

\def\TT{{\mathcal T}}
\def\JJ{{\mathcal J}}

\def\XX{{\mathcal X}}

\def\v{\boldsymbol{v}}
\def\w{\boldsymbol{w}}
\def\B{\boldsymbol{B}}
\def\A{\boldsymbol{A}}

\def\bu{{\boldsymbol{u}}}
\def\bp{{\boldsymbol{p}}}
\def\bv{{\boldsymbol{v}}}
\def\bw{{\boldsymbol{w}}}
\def\bB{{\boldsymbol{B}}}

\def\norm#1#2{\|#1\|_{#2}}

\def\set#1#2{\big\{#1\,:\,#2\big\}}

\def\eps{\varepsilon}

\def\v{\mathbf{v}}

\def\normL2#1#2{\|#1\|_{L^2(#2)}}

\newcommand{\dual}[3][]{#1\langle#2\,,\,#3#1\rangle}

\newcounter{constantsnumber}
\def\namec#1#2{%
 \ifthenelse{\equal{#1}{rel}}{C_{\rm rel}}{%
  \ifthenelse{\equal{#1}{mesh}}{C_{\rm mesh}}{%
  \ifthenelse{\equal{#1}{sz}}{C_{\rm sz}}{%
  \ifthenelse{\equal{#1}{dislocrel}}{C_{\rm dlr}}{%
  \ifthenelse{\equal{#1}{eff}}{C_{\rm eff}}{%
  \ifthenelse{\equal{#1}{main}}{C_{\rm V}}{%
  \ifthenelse{\equal{#1}{opt}}{C_{\rm opt}}{%
  \ifthenelse{\equal{#1}{normequiv}}{C_{\rm norm}}{%
  \ifthenelse{\equal{#1}{reliable}}{C_{\rm rel}}{%
  \ifthenelse{\equal{#1}{efficient}}{C_{\rm eff}}{%
  \ifthenelse{\equal{#1}{dlr}}{C_{\rm dlr}}{%
  \ifthenelse{\equal{#1}{stable}}{C_{\rm stab}}{%
  \ifthenelse{\equal{#1}{reduction}}{C_{\rm red}}{%
   \ifthenelse{\equal{#1}{unibound}}{C_{\rm hot}}{%
    \ifthenelse{\equal{#1}{hotConst}}{C_{\rm hot}}{%
   \ifthenelse{\equal{#1}{inverseK}}{C_{\rm K}}{%
  \ifthenelse{\equal{#1}{refined}}{C_{\rm ref}}{%
  \ifthenelse{\equal{#1}{estconv}}{C_{\rm est}}{%
  \ifthenelse{\equal{#1}{optimal}}{C_{\rm opt}}{%
  \ifthenelse{\equal{#1}{qo}}{C_{\rm qo}}{%
  \ifthenelse{\equal{#1}{mon}}{C_{\rm mon}}{%
  \ifthenelse{\equal{#1}{cea}}{C_{\mbox{\scriptsize C\'ea}}}{%
  \ifthenelse{\equal{#2}{newcounter}}{\refstepcounter{constantsnumber}\label{const#1}}{}C_{\ref{const#1}}}%
}}}}}}}}}}}}}}}}}}}}}}

\newcounter{contractionnumber}
\def\nameq#1#2{%
  \ifthenelse{\equal{#1}{reduction}}{q_{\rm red}}{%
  \ifthenelse{\equal{#1}{estconv}}{q_{\rm est}}{%
  \ifthenelse{\equal{#1}{cea}}{q_{\mbox{\scriptsize C\'ea}}}{%
  \ifthenelse{\equal{#2}{newcounter}}{\refstepcounter{contractionnumber}\label{contraction#1}}{}q_{\ref{contraction#1}}}%
}}}

\def\namer#1#2{%
  \ifthenelse{\equal{#1}{reduction}}{\rho_{\rm red}}{%
  \ifthenelse{\equal{#1}{estconv}}{\rho_{\rm est}}{%
  \ifthenelse{\equal{#1}{cea}}{\rho_{\mbox{\scriptsize C\'ea}}}{%
  \ifthenelse{\equal{#1}{qo}}{\rho_{\mbox{\scriptsize qo}}}{%
  \ifthenelse{\equal{#2}{newcounter}}{\refstepcounter{contractionnumber}\label{contraction#1}}{}\rho_{\ref{contraction#1}}}%
}}}}

\newenvironment{remark}{\medskip\noindent\textbf{Remark.}\ \it}{\qed\smallskip}

\def\T{\mathbb T}

\newcommand\stacktwo[2]{\genfrac{}{}{0pt}{}{#1}{#2}}

\numberwithin{equation}{section}
\numberwithin{theorem}{section}
\begin{document}

\maketitle
\begin{abstract}
	We prove that the  a standard adaptive algorithm for the Taylor-Hood discretization of the stationary Stokes problem converges with optimal rate.
	This is done by developing an abstract framework for quite general problems, which allows us to prove \emph{general quasi-orthogonality} proposed in~\cite{axioms}.
	This property is the main obstacle towards the optimality proof and therefore is the main focus of this work. The key ingredient is a new connection between the mentioned quasi-orthogonality
	and $LU$-factorization of infinite matrices. 
\end{abstract}
\section{Introduction}
We consider an adaptive mixed finite element method (FEM) for the stationary Stokes problem
\begin{align*}
 -\Delta u +\nabla p &=f,\\
 {\rm div} \,u&=0
\end{align*}
with standard Dirichlet boundary conditions in two space dimensions.
We discretize the problem with Taylor-Hood elements and employ a standard adaptive algorithm with D\"orfler marking.
The goal of this work is to provide a missing link in the theory of rate optimality: linear convergence for indefinite problems.

The theory of rate optimal adaptive algorithms for finite element methods originated in the seminal paper~\cite{stevenson07} by Stevenson and was further improved in~\cite{ckns}
by Cascon, Kreuzer, Nochetto, and Siebert. (Note that the notion of rate optimality differs from that of instance optimality for adaptive 
finite element methods introduced in the seminal papers~\cite{bdd,stevensoninstance}.)
These papers prove essentially, that a standard adaptive algorithm of the form
\begin{align*}
 \fbox{Solve}\longrightarrow\fbox{Estimate}\longrightarrow\fbox{Mark}\longrightarrow\fbox{Refine}
\end{align*}
generates asymptotically optimal meshes for the approximation of the solution of a Poisson problem.
The new ideas sparked a multitude of papers applying and extending the techniques to different problems, see e.g.,~\cite{ks,cn} for conforming methods,~\cite{rabus10,BeMao10,bms09,cpr13,mzs10}  
for nonconforming methods,~\cite{LCMHJX,CR2012,HuangXu} for mixed formulations, and~\cite{fkmp,gantumur,affkp,ffkmp:part1,ffkmp:part2} for boundary element methods 
(the list is not exhausted, see also~\cite{axioms} and
the references therein).
All the mentioned results, however, focus on symmetric and definite problems in the sense that the underlying equation induces a symmetric and definite operator. 
The missing link required to extend the theory to non-symmetric and indefinite problems is the quasi-orthogonality and consequently 
the linear convergence of the error. The first is usually an estimate of the form
\begin{align}\label{eq:qointro}
 \norm{u_{\ell+1}-u_{\ell}}{}^2\leq c\norm{u-u_{\ell}}{}^2-C\norm{u-u_{\ell+1}}{}^2 +   \text{terms of higher order}
\end{align}
with $c\approx C \approx 1$.
However, for indefinite and non-symmetric problems such an estimate does not seem to hold.
The first proof of rate optimality for a non-symmetric problem which does not rely on additional assumptions on the initial mesh is given in~\cite{nonsymm} for a general second order elliptic operator with non-vanishing diffusion
coefficient of the form
\begin{align*}
-{\rm div}(A\nabla u) + b\cdot \nabla u + cu=f.
\end{align*}
This approach, however, relies heavily on the fact that the non-symmetric part of the operator $ (b\cdot \nabla u + cu)$ is only a compact perturbation 
(one differentiation instead of two for the diffusion part). The first optimality proof of a strongly non-symmetric problem was given in the recent work~\cite{fembemopt} for
a finite-element/boundary-element discretization of a transmission problem. 
The proof relies on a novel form of quasi-orthogonality introduced in~\cite{axioms,nonsymm} called \emph{general quasi-orthogonality}.
Heuristically, instead of requiring~\eqref{eq:qointro} to hold for all $\ell\in\N$, it must only be true in a cumulative sense. 
This notion of quasi-orthogonality is the last missing building block to apply the theory developed in~\cite{stevenson07,ckns} and culminating in~\cite{axioms}
to prove linear convergence and ultimately rate optimal convergence.

The present work aims to generalize the approach from~\cite{fembemopt} to include indefinite problems.
While the work is concerned with the Stokes problem, the applied methods are quite general and may be useful for other indefinite problems. 

Currently available convergence and optimality theory for the Stokes problem is building on the seminal works~\cite{dahlke,nochettostokes}.
For certain non-standard (Uzawa type) algorithms for the Stokes problem, the work~\cite{stevensonstokes} proves optimal convergence. For nonconforming finite element methods, rate optimality and convergence has 
been investigated and achieved in~\cite{ncstokes1,ncstokes2,ncstokes3} by use of the fact that the Crouzeix-Raviart discretization of the Stokes problem essentially transforms it into a symmetric and elliptic problem. 

For the standard Taylor-Hood element, the first proofs of adaptive convergence were presented in~\cite{msv,siebert}, while the first a~posteriori error estimator was presented in~\cite{verfstokes}.
The work~\cite{gantstokes} gives an optimality proof under the assumption that \emph{general quasi-orthogonality} is satisfied.
This assumption is verified in the present work.

Since the level of technicality is already considerably high in the present two dimensional case, we refrain from presenting the general
case. However, there doesn't seems to be any inherent barrier and the proof techniques are expected to transfer.

\medskip
The remainder of the work is organized as follows: We present the setting and the main theorem in Section~\ref{section:gen}, discuss the key ideas of the proof in Section~\ref{sec:key}, construct a suitable Riesz basis in Section~\ref{section:basis}, and prove the main result in Section~\ref{sec:proofA1A4}. We conclude the work with Section~\ref{sec:tech} containing technical results which are required for the proof.

\section{General assumptions and the Stokes problem}\label{section:gen}
\subsection{Preliminaries and notation}
In the following, $\Omega\subseteq \R^2$ is a polygonal domain with Lipschitz boundary $\partial\Omega$. 
The restriction to $\R^2$ is not fundamental but simplifies some constructions later on in the technical sections.
Given a Lipschitz domain $\omega\subseteq \R^2$, we denote by $H^s(\omega)$ the usual Sobolev spaces for $s\geq 0$.
For non-integer values of $s$, we use real interpolation to define $H^s(\omega)$. 
Their dual spaces $\widetilde H^{-s}(\omega)$ are defined by extending the $L^2$-scalar product. 
We denote by $H^1_0(\Omega)$ the $H^1$-functions with vanishing trace and by $H^{-1}(\Omega)$ the dual space. The subscript $\star$ denotes vanishing mean 
of the functions in the given space, i.e., $L^2_\star(\Omega):=\set{v\in L^2(\Omega)}{\dual{v}{1}_\Omega=0}$
(the definition generalizes to $\widetilde H^{-s}(\Omega)$ in a straightforward fashion). Finally, $\PP^p(\omega)$ denotes
the polynomials of total degree less or equal to $p$. 

By $\norm{\cdot}{\ell_2}$, we denote the Euclidean distance of vectors in $\R^n$ as well as the $\ell_2$-sequence norm. 
The norm $\norm{\cdot}{\ell_2}$ induces the spectral norm for (infinite) matrices denoted by $\norm{\cdot}{2}$. The symbol $I$ denotes the identity matrix. The size of the matrix is always determined by the context.

\subsection{Mesh refinement}\label{section:mesh}
Let $\TT_0$ be a triangulation of $\Omega$. Given two triangulations $\TT,\TT^\prime$, we write $\TT^\prime={\rm refine}(\TT,\MM)$ for some $\MM\subseteq \TT$ if 
$\TT^\prime$ is generated from $\TT$ by refinement of all $T\in\MM$ (and subsequent mesh-closure) via newest vertex bisection. In contrast to the general case where a certain initial labelling of the edges is required to obtain an optimal refinement algorithm~\cite{bdd}, the paper~\cite{nvb2d} shows that in two space dimensions those initial conditions are not necessary. 
We write $\TT^\prime\in{\rm refine}(\TT)$
if $\TT^\prime$ is generated from $\TT$ by a finite number $j\in\N$ of iterated newest-vertex-bisection  refinements. To keep mesh conformity, one usually has to refine more elements then initially marked for refinement (mesh-closure). The number of extra refinements is bounded cumulatively in the sense that
\begin{align*}
 \#\TT^\prime-\#\TT_0\lesssim \sum_{i=0}^j\#\MM_j,
\end{align*}
where $\MM_j$ are elements marked for refinement in the individual steps of the refinement procedure, see~\cite{bdd,stevenson08,nvb2d}.
We denote the set of all possible refinements by $\T:={\rm refine}(\TT_0)$.
Given  $\omega\subseteq\Omega$, we call $\TT^\prime|_\omega$ a local refinement of $\TT$,
if there exists $\TT^{\prime\prime}\in {\rm refine}(\TT)$ such that $\TT^\prime|_\omega=\TT^{\prime\prime}|_\omega$.
Given $T\in\TT$ for some $\TT\in\T$, ${\rm level}(T)$ denotes the number of bisections necessary to generate $T$ from a parent element in $\TT_0$.

We define $\NN(\TT)$ as the set of nodes of $\TT$ and $\EE(\TT)$ as the set of edges of $\TT$.
For any triangulation, we define 
\begin{align*}
\PP^p(\TT)&:=\set{v\in L^2(\Omega)}{v|_T\in\PP^p(T),\,T\in\TT},\\
\SS^p(\TT)&:=\PP^p(\TT)\cap H^1(\Omega),\\
\SS^p_0(\TT)&:=\PP^p(\TT)\cap H^1_0(\Omega),\\
\SS^p_\star(\TT)&:=\PP^p(\TT)\cap H^1_\star(\Omega).
\end{align*}
We define $h_\TT\in\PP^0(\TT)$ as the mesh-size function by $h_\TT|_T:={\rm diam}(T)$ for all $T\in\TT$.

\medskip

Given a subset $\Omega^\prime\subseteq \Omega$, we define the patch
\begin{align*}
 \omega(\Omega^\prime,\TT):=\set{T_1\in\TT}{\exists T_2\in\TT,\,T_1\cap T_2\neq \emptyset,\, \int_{T_2\cap \Omega^\prime}1\,dx>0}.
\end{align*}
The extended patches $\omega^k(\Omega^\prime,\TT)$ are defined iteratively by
\begin{align*}
 \omega^1(\Omega^\prime,\TT):=\omega(\Omega^\prime,\TT),\quad\text{and}\quad  \omega^k(\Omega^\prime,\TT):=\omega(\bigcup\omega^{k-1}(\Omega^\prime,\TT),\TT).
\end{align*}

\subsection{The Stokes problem and the Taylor-Hood element}
Our main goal here is to prove general quasi-orthogonality and hence optimality of an adaptive algorithm
for the stationary Stokes problem, which reads
\begin{align}\label{eq:stokes}
\begin{split}
 -\Delta \bu +\nabla \bp &=f\quad\text{in }\Omega,\\
 {\rm div} \bu &= 0\quad\text{in }\Omega,\\
 \bu&=0\quad\text{on }\partial\Omega,\\
\mbox{$\int_\Omega$} \bp\,dx&=0
\end{split}
 \end{align}
for given given functions $f\in L^2(\Omega)$ with weak solutions $\bu\in H_0^1(\Omega)^2$ and $\bp\in L^2(\Omega)$. 
We define the space $\XX:=H^1_0(\Omega)^2\times L^2_\star(\Omega)$.

The weak formulation of~\eqref{eq:stokes} reads: Find $(\bu,\bp)\in\XX$ such that all $(v,q)\in\XX$ satisfy
\begin{align}\label{eq:weak}
a((\bu,\bp),(v,q)):= \int_\Omega \nabla \bu\cdot \nabla v\,dx - \int_\Omega \bp\,{\rm div}v\,dx -\int_\Omega q\,{\rm div}\bu\,dx = \int_\Omega fv\,dx.
\end{align}
For the purpose of discretization, we choose standard Taylor-Hood elements defined by
\begin{align*}
 \XX_\TT:= \SS^2_0(\TT)^2 \times \SS^1_\star(\TT). 
\end{align*}
Thus, the Galerkin formulation reads: Find $(\bu_\TT,\bp_\TT)\in\XX_\TT$ such that all $(v,q)\in\XX_\TT$ satisfy
\begin{align}\label{eq:galerkin}
a((\bu_\TT,\bp_\TT),(v,q)):= \int_\Omega \nabla \bu_\TT\cdot \nabla v\,dx - \int_\Omega \bp_\TT\,{\rm div}v\,dx -\int_\Omega q\,{\rm div}\bu_\TT\,dx = \int_\Omega fv\,dx.
\end{align}

We use a locally equivalent variation proposed in~\cite{gantstokes}  of the classical error estimator proposed by Verf\"urth~\cite{verfstokes}, 
i.e., for all $T\in\TT$ define
 \begin{align*}
\eta_T(\TT)^2&:={\rm diam}(T)^2\norm{f+\Delta \bu_\TT-\nabla \bp_\TT}{L^2(T)}^2 +{\rm diam}(T)\norm{[\partial_n \bu_\TT]}{L^2(\partial T\cap\Omega)}^2\\
&\qquad\qquad + {\rm diam}(T)\norm{{\rm div}(\bu_\TT)|_T}{L^2(\partial T)}^2,
\end{align*}
where $[\cdot]$ denotes the jump across an edge of $\TT$. (Note that there are also other error estimators which could be used here, e.g., those in~\cite{msv}.)
The overall estimator reads
\begin{align*}
 \eta(\TT):=\Big(\sum_{T\in\TT}\eta_T(\TT)^2\Big)^{1/2}\quad\text{for all }\TT\in\T
\end{align*}
and satisfies upper and lower error bounds, i.e.,
\begin{align}\label{eq:releff}
 C_{\rm rel}^{-1}\norm{\bu-\bu_\TT,\bp-\bp_\TT}{\XX} \leq \eta(\TT)^2\leq
 C_{\rm eff}\Big(\norm{\bu-\bu_\TT,\bp-\bp_\TT}{\XX}^2 + {\rm osc}(\TT)^2\Big)^{1/2},
\end{align}
where the data oscillation term reads ${\rm osc}(\TT)^2:= \min_{g\in \PP^0(\TT)}\sum_{T\in\TT} {\rm diam}(T)^2\norm{f-g}{L^2(T)}^2$.
We define $\XX_\ell:=\SS^{2}(\TT_\ell)^2\times \SS^1_\star(\TT_\ell)$ to denote the adaptively generated spaces from Algorithm~\ref{algorithm}, below.

\subsection{Inf-sup condition and some notation}
To simplify presentation, we collect velocity and pressure in one variable, i.e., $u=(\bu,\bp)\in\XX$ and $u_\TT=(\bu_\TT,\bp_\TT)\in\XX_\TT$.
According to~\cite{brezzi-fortin}, the Stokes problem $a(\cdot,\cdot)\colon\XX\times\XX\to\R$ satisfies
 \begin{align}\label{eq:infsup}
  \inf_{u\in\XX}\sup_{v\in\XX}\frac{a(u,v)}{\norm{u}{\XX}\norm{v}{\XX}}>0. 
 \end{align}
If $\TT_0$ contains at least three triangles,~\cite{boffi} proves even
 \begin{align}\label{eq:discinfsup}
  \inf_{\TT\in\T}\inf_{u\in\XX_\TT}\sup_{v\in\XX_\TT}\frac{a(u,v)}{\norm{u}{\XX}\norm{v}{\XX}}>0.
 \end{align}
 This shows existence and uniqueness of the weak solutions defined above. In the new notation, we have $u\in\XX$ 
 as well as the discrete solution $u_\TT\in\XX_\TT$ by 
\begin{align}\label{eq:solutions}
 a(u,v)=\dual{f}{v}\quad\text{for all }v\in\XX\quad\text{and}\quad a(u_\TT,v)=\dual{f}{v}\quad\text{for all }v\in\XX_\TT.
 \end{align}

\subsection{Adaptive algorithm}

Given a triangulation $\TT\in\T$, we assume that we can compute the error estimator $\eta(\TT)=\sqrt{\sum_{T\in\TT}\eta_T(\TT)^2}$ exactly.
Due to technical reasons, we have to restrict to adaptive triangulations with mild grading in the sense that there exists $D_{\rm grad}\in\N$ such that
\begin{align}\label{eq:graded}
 |{\rm level}(T)-{\rm level}(T^\prime)|\leq 1\quad\text{for all }T^\prime \in \omega^{D_{\rm grad}}(T,\TT)\text{ and all }T\in\TT.
\end{align}
This condition seems to be required for the present proof and also appears in~\cite{l2opt} to prove optimal convergence in the $L^2$-norm. 
By $\T_{\rm grad}\subseteq \T$, we denote all triangulations which satisfy~\eqref{eq:graded} for a given $D_{\rm grad}\in\N$. The result~\cite[Lemma~2.3]{fembemopt} shows that the restriction
does not alter the optimal convergence rate. Numerical experiments suggest that the restriction is not even necessary for optimal convergence rate, and thus
might just be an artifact of the proof. In the following, we assume that $D_{\rm grad}$ is sufficiently large to satisfy all the conditions in the proofs below.

We assume that the sequence $(\TT_\ell)_{\ell\in\N_0}\subset \T^{\rm grad}$ is generated by an adaptive algorithm of the form:
\begin{algorithm}[H]\label{algorithm}
Input: $\ell=0$, $\TT_0$, $D_{\rm grad}\in\N$, $0< \theta\leq 1$, $f\in \XX^\star$.\\
For $\ell=0,1,\ldots$ do:
 \begin{enumerate}
  \item Compute $u_\ell\in\XX_\ell$ as the unique solution of
  \begin{align}\label{eq:adapsol}
   a(u_\ell,v)=\dual{f}{v}\quad\text{for all }v\in\XX_\ell.
   \end{align}
  \item Compute error estimator $\eta_T(\TT_\ell)$ for all $T\in\TT_\ell$.
  \item Mark set of minimal cardinality $\MM_\ell\subseteq \TT_\ell$ such that
  \begin{align}\label{eq:doerfler}
   \sum_{T\in\MM_\ell}\eta_T(\TT_\ell)\geq \theta \sum_{T\in\TT_\ell}\eta_T(\TT_\ell).
  \end{align}
  \item Refine at least the elements $\MM_\ell$ of $\TT_\ell$ to obtain $\TT_{\ell+1}$.
  \item Refine additional elements to ensure that $\TT_{\ell+1}$ satisfies~\eqref{eq:graded} (see, e.g.,~\cite[Section~A.3]{l2opt}
  for a valid mesh-refinement algorithm).
 \end{enumerate}
Output: sequence of meshes $\TT_\ell$ and corresponding solutions $u_\ell$.
\end{algorithm}

\subsection{Rate optimality}
We aim to analyze the best possible algebraic convergence rate which can be obtained by the adaptive algorithm.
This is mathematically characterized as follows: For the exact solution $u\in \XX$, we define an approximation class $\A_s$ by
\begin{align}\label{def:approxclass}
 u \in \A_s
 \quad\overset{\rm def.}{\Longleftrightarrow}\quad
\norm{u}{\A_s}:= \sup_{N\in\N}\min_{\stacktwo{\TT\in\T}{ \#\TT-\#\TT_0\leq N}}N^s\eta(\TT) < \infty.
\end{align}
 By definition, a convergence rate $\eta(\TT) = \OO(N^{-s})$ is theoretically possible if the optimal meshes are chosen.
 In view of mildly graded triangulations, we define
 \begin{align*}
   u \in \A_s^{\rm grad}
 \quad\overset{\rm def.}{\Longleftrightarrow}\quad
\norm{u}{\A_s^{\rm grad}}:= \sup_{N\in\N}\min_{\stacktwo{\TT\in\T_{\rm grad} }{ \#\TT-\#\TT_0\leq N}}N^s\eta(\TT) < \infty.
 \end{align*}
In~\cite[Lemma~2.3]{fembemopt}, we show that in many situations (including the present setting) $\A_s^{\rm grad}=\A_s$ 
for algebraic rates $s>0$. However, condition~\eqref{eq:graded} influences $\norm{u}{\A_s^{\rm grad}}$ and leads to a slightly non-standard AFEM.
In the spirit of~\cite{axioms}, rate optimality of the adaptive algorithm means that there exists a constant $C_{\rm opt}>0$ such that
\begin{align*}
 C_{\rm opt}^{-1}\norm{u}{\A_s}\leq\sup_{\ell\in\N_0}\frac{\eta(\TT_\ell)}{(\#\TT_\ell-\#\TT_0+1)^{-s}}\leq C_{\rm opt}\norm{u}{\A_s},
\end{align*}
for all $s>0$ with $\norm{u}{\A_s}<\infty$.

\subsection{An abstract theory of rate optimality}\label{sec:axioms}
 As proved in~\cite{axioms}, we need to check the assumptions~(A1)--(A4) to ensure rate optimality for a given adaptive algorithm: There exist constant $C_{\rm red}$, $C_{\rm stab}$, $C_{\rm qo}$, $C_{\rm dlr}$,
$C_{\rm ref}>0$,
and $0\leq q_{\rm red}<1$
such that

\renewcommand{\theenumi}{{\rm A\arabic{enumi}}}%
\newcounter{subterm}%
\begin{enumerate}
\item \label{A:stable}\textbf{Stability on non-refined elements}: For all refinements $\widehat\TT\in\T$ of a triangulation $\TT\in\T$,
for all subsets $\mathcal{S}\subseteq\TT\cap\widehat\TT$ of non-refined elements, it holds that
\begin{align*}
\Big|\Big(\sum_{T\in\mathcal{S}}\eta_T(\widehat\TT)^2\Big)^{1/2}-\Big(\sum_{T\in\mathcal{S}}\eta_T(\TT)^2\Big)^{1/2}\Big|\leq C_{\rm stab}\,\norm{u_\TT- u_{\widehat\TT}}{\XX}.
\end{align*}
\item \label{A:reduction}\textbf{Reduction property on refined elements}: Any refinement $\widehat\TT\in\T$ of a triangulation $\TT\in\T$ satisfies
\begin{align*}
\sum_{T\in\widehat\TT\setminus\TT}\eta_T(\widehat\TT)^2\leq q_{\rm red} \sum_{T\in\TT\setminus\widehat\TT}\eta_T(\TT)^2 + C_{\rm red}\norm{u_\TT- u_{\widehat\TT}}{\XX}^2.
\end{align*}
\item\label{A:qosum}\textbf{General quasi-orthogonality}: For one sufficiently small $\eps\geq 0$ (see~\cite[Section~3]{axioms} for details) the output of Algorithm~\ref{algorithm}
satisfies for all $\ell,N\in\N_0$
\begin{align}\label{eq:qosum}
 \sum_{k=\ell}^{\ell+N} \Big(\norm{u_{k+1}-u_k}{\XX}^2-C_{\rm qo}\eps\norm{u-u_k}{\XX}^2\Big)\leq C_{\rm qo} \norm{u-u_\ell}{\XX}^2.
\end{align}
\item\label{A:dlr}\textbf{Discrete reliability}:
For all refinements $\widehat\TT\in\T$ of a triangulation $\TT\in\T$,
there exists a subset $\RR(\TT,\widehat\TT)\subseteq\TT$ with $\TT\backslash\widehat\TT \subseteq\RR(\TT,\widehat\TT)$ and
$|\RR(\TT,\widehat\TT)|\le C_{\rm ref}|\TT\backslash\widehat\TT|$ such that
\begin{align*}
 \norm{u_{\widehat\TT}-u_\TT}{\XX}^2
 \leq C_{\rm dlr}^2\sum_{T\in\RR(\TT,\widehat\TT)}\eta_T(\TT)^2.
\end{align*}
\end{enumerate}

While the assumptions~(A1), (A2), and (A4) are already known in various forms in the literature, the general quasi-orthogonality~(A3) seems
to be the main obstacle for the optimality proof.

\subsection{The main result}

The following result shows rate optimality of the adaptive algorithm and is the main result of the paper. 

\begin{theorem}[Optimality of the adaptive algorithm]\label{thm:opt}
Given sufficiently small $\theta>0$ and sufficiently large $D_{\rm grad}\geq 1$, 
Algorithm~\ref{algorithm} applied to the stationary Stokes problem as described above guarantees rate-optimal convergence, i.e., there exists a constant $C_{\rm opt}>0$ such that
\begin{align*}
 C_{\rm opt}^{-1}\norm{u}{\A_s^{\rm grad}}\leq\sup_{\ell\in\N_0}\frac{\eta(\TT_\ell)}{(\#\TT_\ell-\#\TT_0+1)^{-s}}\leq C_{\rm opt}\norm{u}{\A_s^{\rm grad}},
\end{align*}
as well as
\begin{align*}
 C_{\rm eff}^{-1}C_{\rm opt}^{-1}\norm{u}{\A_s^{\rm grad}}\leq\sup_{\ell\in\N_0}\frac{\sqrt{\norm{u-u_\ell,p-p_\ell}{\XX}^2 + {\rm osc}(\TT_\ell)^2}}{(\#\TT_\ell-\#\TT_0+1)^{-s}}\leq (1+C_{\rm rel})C_{\rm opt}\norm{u}{\A_s^{\rm grad}},
\end{align*}
for all $s>0$ with $\norm{u}{\A_s}<\infty$.
\end{theorem}
\begin{proof}
Note that both statements are equivalent due to the upper and lower bounds in~\eqref{eq:releff}. To prove the first statement,
we have to specify a mesh-refinement strategy which ensures~\eqref{eq:graded} and fits into the framework of~\cite[Section~2.4]{axioms}.
To that end, we use the strategy specified in~\cite[Section~A.3]{l2opt} (note that the condition in~\cite[Section~A.3]{l2opt} and~\eqref{eq:graded} are equivalent up to shape regularity).
Then, the result follows immediately from~\cite[Theorem~4.1]{axioms} and~\cite[Lemma~2.3]{fembemopt}, after we prove the assumptions (A1)--(A4) in Section~\ref{sec:proofA1A4}. (Note that instead of this proof, we could also use Theorem~\ref{thm:qosum} together with~\cite{gantstokes}, in which general quasi-orthogonality is assumed in order to prove optimality.)
\end{proof}
\begin{remark}
 The condition on $\theta>0$ being sufficiently small can be quantified in terms of the constants in (A1)--(A4) (see~\cite{axioms}). However,
 in general one does not know the constants and hence has to guess $\theta$. Practical examples show that a choice of $\theta=1/2$ is very robust in many cases. This is in stark contrast to the maximum marking strategy~\cite{bdd, stevensoninstance}, where one can prove that any marking parameter leads to optimality. The maximum strategy, however, is not proven to produce work optimal algorithms. We also would like to mention the work~\cite{ksinst} on the Stokes problem, which proves instance optimality for a maximum marking strategy. This work, however, relies heavily on the fact that the Crouzeix-Raviart discretization essentially turns the Stokes problem into a symmetric and elliptic problem.
\end{remark}

\section{The key ideas of the proof}\label{sec:key}
This section presents the two insights which enable the proof of general quasi-orthogonality and thus optimal convergence. 
The first one in Section~\ref{sec:LUqo} is that general quasi-orthogonality is strongly related to the existence 
and boundedness of the $LU$-factorization of an infinite stiffness matrix. 
The second one in Section~\ref{section:explu} is that a certain exponential decay property of infinite matrices
guarantees the boundedness of the $LU$-factors. 

\medskip

Sketch of the proof: The first idea to prove~(A3) and hence the main theorem (Theorem~\ref{thm:opt}) was as follows: Find an $\XX$-orthogonal basis $\bv_1,\bv_2,\ldots$ of $\XX$
such that all the adaptive spaces generated by Algorithm~\ref{algorithm} are spanned by it, i.e., $\XX_\ell={\rm span}\{\bv_1,\ldots,\bv_{N_\ell}\}$ 
for all $\ell\in\N$ and numbers $N_\ell\in\N$. Then, consider the infinite stiffness matrix $A\in\R^{\N\times\N}$, $A_{ij}:=a(\bv_j,\bv_i)$
and rewrite the original problem as an infinite matrix problem. Section~\ref{sec:LUqo} shows that a (block)-$LU$-factorization $A=LU$ with bounded factors implies general quasi-orthogonality~(A3).
A bound on the factors $L$ and $U$ would hence conclude the proof as we already mentioned above that the other requirements (A1), (A2), and (A4) are well-known.

\medskip

The main problem here is that the $LU$-factorization results of Section~\ref{section:lu} are not strong enough to show that the matrix $A$ has a
bounded $LU$-factorization. Hence, we require more structure (bandedness, exponential decay) of the infinite stiffness matrix and therefore have
to replace the orthogonal basis $\bv_1,\bv_2,\ldots$
by a hierarchical (wavelet-type) Riesz basis which is local in a certain sense. Unfortunately, this introduces some technical difficulties as well as the need for mildly graded meshes.
It seems that the topic of the stability of the $LU$-factorization is very complex and not well-understood in the literature. Advances in this direction could significantly reduce the length of this work.

\medskip

One can say that Sections~\ref{sec:LUqo} and~\ref{section:explu} contain the most interesting 
ideas and are problem independent in the sense that they only require a bilinear form and a conforming method to be applicable. This means that the main ideas of this work are expected to carry over to other challenging problems in the domain of adaptive algorithms.  
The remaining sections (particularly Section~\ref{section:basis}) solve the technical details and are tailored to the space $\XX$ in which the Stokes problem is formulated.
The techniques are quite general and might also be useful for other problem spaces. However, the proofs are not likely to carry over directly.

\subsection{LU-factorization and general quasi-orthogonality}\label{sec:LUqo}
Before we proceed, we need to clarify the notion of block-structure of matrices.
\begin{definition}\label{def:blockstruct}
Given a  block structure in the sense that there exist numbers $n_1,n_2,\ldots\in\N$ with $n_1=1$ and $n_i<n_j$ for all $i\leq j$, we denote matrix blocks by
\begin{align*}
M(i,j):= M|_{\{n_i,\ldots,n_{i+1}-1\}\times \{n_j,\ldots,n_{j+1}-1\}}\in\R^{(n_{i+1}-n_i)\times (n_{j+1}-n_j)}.
\end{align*}
By $M[k]\in \R^{(n_{k+1}-1) \times (n_{k+1}-1)}$, we denote the restriction of $M$ to the first $k\times k$ blocks.
\end{definition}
A matrix $M\in\R^{\N\times\N}$ has an $LU$-factorization if
$M=LU$ for matrices $L,U\in\R^{\N\times\N}$ such that
\begin{align*}
L_{ii}=1\text{ and } L_{ij}=U_{ji}=0\quad\text{for all } i,j\in\N,\,i< j.
\end{align*}
Given a block-structure in the sense of Definition~\ref{def:blockstruct}, we say that $M$ has a block-$LU$-factorization if $A=LU$ and $L$ and $U$ satisfy
\begin{align*}
 L(i,i)=I\text{ and } L(i,j) = U(j,i)=0\quad\text{for all } i,j\in\N,\,i< n_r \leq  j\text{ for some }r\in\N.
\end{align*}
A sufficient condition for the existence of a block-$LU$-factorization is
the matrix inf-sup condition 
	\begin{align}\label{eq:matinfsup}
	\sup_{k\in\N}\norm{(M[k])^{-1}}{2}=C_{\rm inf-sup}<\infty.
	\end{align}

The following result connects existence of bounded $LU$-factors with general quasi-orthogonality.
\begin{theorem}\label{thm:luqo2}
	Let there exist a Riesz bases $(w_n)_{n\in\N}$ of $\XX$ and a constant $C>0$ such that all $x=\sum_{n\in\N} \lambda_n w_n\in \XX$ satisfy
	\begin{align}\label{eq:schauder}
	C^{-1}\norm{x}{2}^2\leq \sum_{n\in\N} \lambda_n^2\leq C\norm{x}{2}^2
	\end{align}
	and there holds $\XX_\ell
	={\rm span}\set{w_n}{n\in\{1,\ldots,N_\ell\}}$ for some constants $N_\ell\in\N$ with $N_{\ell}< N_{\ell+1}$ (a block structure in the sense of Definition~\ref{def:blockstruct}).
	Assume that for some $\eps>0$, there exists $M^\eps\in\R^{\N\times\N}$ such that $M_{ij}:=a(w_j,w_i)$, $M\in \R^{\N\times\N}$ satisfies $\norm{M-M^\eps}{2}\leq \eps$.
	If $M$ and $M^\eps$ satisfy~\eqref{eq:matinfsup} and $M^\eps=LU$ has a block-$LU$-factorization such that $U,U^{-1}\colon \ell_2\to \ell_2$ are bounded operators, then there holds general quasi-orthogonality~\eqref{eq:qosum}.
	The constant $C_{\rm qo}>0$ depends only on the basis $(w_n)$, $a$, $C_{\rm inf-sup}$, $C$, and on the norms of $U, U^{-1}$.
\end{theorem}
We postpone the proof of Theorem~\ref{thm:luqo2} to the end of this section.

\begin{lemma}\label{lem:luqo}
	Let $(w_n)_{n\in\N}$ denote the Riesz basis from the statement of Theorem~\ref{thm:luqo2}. 
	If $M_{ij}:=a(w_j,w_i)$ and $M$ has a block-$LU$-factorization with bounded operators $U,U^{-1}\colon \ell_2\to \ell_2$,
	then there holds general quasi-orthogonality~\eqref{eq:qosum} even with $\eps=0$.
	The constant $C_{\rm qo}$ depends only on the basis $(w_n)$, $C_{\rm inf-sup}$, the norm of $a(\cdot,\cdot)$, $C$,  and $\XX$.
\end{lemma}
\begin{proof}
In the following proof, we identify vectors $\lambda\in \R^N$ with vectors $\lambda\in\R^\N$ by adding zero entries on the right, i.e.,
$\lambda = (\lambda_1,\ldots,\lambda_N, 0,\ldots)$. The same is done with matrices $M\in\R^{N\times N}$ by embedding them into the upper-left corner of an infinite zero matrix.
	By~\eqref{eq:schauder} and with~\eqref{eq:infsup}--\eqref{eq:discinfsup}, the matrix $M$ is bounded and satisfies~\eqref{eq:matinfsup}.
	Let $u=\sum_{n=1}^\infty \lambda_n w_n$ and $u_\ell=\sum_{n=1}^\infty \lambda(\ell)_n w_n$.
	With $\lambda:=(\lambda_1,\lambda_2,\ldots)$ and $\lambda(\ell):=(\lambda(\ell)_1,\lambda(\ell)_{2},\ldots,\lambda(\ell)_{N_\ell})$, $F:=(\dual{f}{w_1},\dual{f}{w_2},\ldots)\in\R^\N$,
	there holds
	\begin{align}\label{eq:l2sol}
	M\lambda=F\quad\text{\rm and}\quad M[\ell]\lambda(\ell)=F[\ell]\quad\text{\rm for all }\ell\in\N,
	\end{align}
	where $F[\ell]:=(F_1,\ldots,F_{N_\ell})$.  Moreover, there holds $M[\ell]=L[\ell]U[\ell]$ for any block-$LU$-factorization.
	Due to the lower-triangular structure of $L$ (note that $L(i,i)=I$), there holds for all $1\leq i\leq N_\ell$ that
	\begin{align*}
	(L[\ell]U[\ell]\lambda(\ell))_i=(M[\ell]\lambda(\ell))_i=F[\ell]_i=F_i=(M\lambda)_i=(LU\lambda)_i=(L[\ell]U\lambda)_i.
	\end{align*}
	Since $L$ and hence also $L[\ell]$ is regular, this shows that $(U\lambda)_i=(U[\ell]\lambda(\ell))_i$ for all $1\leq i\leq N_\ell$.
	Moreover, there holds $U[\ell]\lambda(\ell)=U\lambda(\ell)$ due to the block-upper triangular structure of $U$. Altogether, this proves
	\begin{align*}
	(U\lambda)_i=(U\lambda(\ell))_i\quad\text{for all }1\leq i\leq N_\ell\quad\text{and}\quad (U\lambda(\ell))_i=0\quad\text{for all }i>N_\ell.
	\end{align*}
	Hence, we have, by use of the boundedness of $U$ and $U^{-1}$ and~\eqref{eq:schauder}, that
	\begin{align*}
	\norm{u_{k+1}-u_k}{\XX}&\simeq \norm{\lambda(k+1)-\lambda(k)}{\ell_2}\\
	&\simeq \norm{U\lambda(k+1)-U\lambda(k)}{\ell_2}=
	\norm{(U\lambda)|_{\{N_k,\ldots,N_{k+1}-1\}}}{\ell_2}.
	\end{align*}
	This shows
	\begin{align*}
	\sum_{k=\ell}^\infty \norm{u_{k+1}-u_k}{\XX}^2&\simeq \sum_{k=\ell}^\infty \norm{(U\lambda)|_{\{N_k,\ldots,N_{k+1}-1\}}}{\ell_2}^2= \norm{(U\lambda)|_{\{N_\ell,\ldots\}}}{\ell_2}^2\\
	&=\norm{U\lambda - U\lambda(\ell)}{\ell_2}^2\simeq \norm{\lambda - \lambda(\ell)}{\ell_2}^2\simeq \norm{u-u_\ell}{\XX}^2.
	\end{align*}
	Hence, we conclude the proof.
\end{proof}

\begin{proof}[Proof of Theorem~\ref{thm:luqo2}]
In the following, instead of function spaces, we consider a $\ell_2$-setting of infinite vectors and apply Lemma~\ref{lem:luqo} to it, i.e.,
$\XX=\ell_2$.
	With~\eqref{eq:schauder}, the boundedness of $a(\cdot,\cdot)$ implies $\norm{M}{2}+\norm{M^\eps}{2}<\infty$.
	With the notation from the proof of Lemma~\ref{lem:luqo}, we apply Lemma~\ref{lem:luqo} to the bilinear form $a^\eps\colon \ell_2\times\ell_2\to\R$, $a^\eps(x,y):=\dual{M^\eps x}{y}_{\ell_2}$
	and $f^\eps:=M^\eps \lambda\in\ell_2$ and the spaces $\XX_\ell:=\set{x\in\ell_2}{x_i=0\text{ for }i>N_\ell}$. Boundedness of $M^\eps$ together with~\eqref{eq:matinfsup} imply boundedness and 
	the inf-sup condition~\eqref{eq:infsup} for $a^\eps(\cdot,\cdot)$.
	We use the $\ell_2$ unit vectors as the Riesz basis to obtain with Lemma~\ref{lem:luqo}
	\begin{align}\label{eq:lambdaqo}
	\sum_{k=\ell}^\infty \norm{\lambda^\eps(k+1)-\lambda^\eps(k)}{\ell_2}^2 \lesssim \norm{\lambda - \lambda^\eps(\ell)}{\ell_2}^2.
	\end{align}
	Here, we used that $\lambda^\eps=\lambda$ by definition of $f^\eps$.
	
	As in the previous proof, we identify vectors in $\R^n$ with vectors in $\R^\N$ by adding zeros to the right. Then, there holds with~\eqref{eq:infsup}
	\begin{align*}
	\norm{\lambda^\eps(k)-\lambda(k)}{\ell_2}&\lesssim \sup_{\stacktwo{\mu\in \XX_k}{\norm{\mu}{\ell_2}=1}}a^\eps(\lambda^\eps(k)-\lambda(k),\mu)\\
	&= \sup_{\stacktwo{\mu\in \XX_k}{\norm{\mu}{\ell_2}=1}}a^\eps(\lambda-\lambda(k),\mu)\\
	&= \sup_{\stacktwo{\mu\in \XX_k}{\norm{\mu}{\ell_2}=1}}\dual{(M^\eps-M)(\lambda-\lambda(k))}{\mu}_{\ell_2}\\
	&\leq \norm{M^\eps-M}{2}\norm{\lambda-\lambda(k)}{\ell_2}.
	\end{align*}
	Hence, we have
	\begin{align*}
	\big|\norm{\lambda^\eps(k+1)-\lambda^\eps(k)}{\ell_2}-\norm{\lambda(k+1)-\lambda(k)}{\ell_2}\big|\lesssim\eps(\norm{\lambda-\lambda(k)}{\ell_2}+\norm{\lambda-\lambda(k+1)}{\ell_2}),
	\end{align*}
	where the hidden constant is independent of $\eps>0$.
	With~\eqref{eq:lambdaqo} and~\eqref{eq:schauder}, this concludes
	\begin{align*}
	 \sum_{k=\ell}^\infty \Big(\norm{u_{k+1} - u_k}{\XX}^2-2C^2\eps\norm{u-u_k}{\XX}^2\Big)\lesssim C^2\norm{u-u_\ell}{\XX}^2
	\end{align*}
and hence the assertion follows.
\end{proof}

\subsection{Jaffard class matrices and $LU$-factorization}\label{section:explu}
This section discusses the second main insight of this work, namely that a particular mixture of exponential decay and block-bandedness ensures bounded $LU$-factorization. (This may not be viewed as new in the general matrix context, however, the available literature does not seem to be general enough to be applicable to this problem.)

Jaffard class matrices generalize the notion of matrices which decay exponentially away from the diagonal. 
The generalization allows to replace the distance $|i-j|$ between indices by a general metric $d(i,j)$. This class was introduced and analyzed in~\cite{rescue}.

\begin{definition}[Jaffard class]\label{def:wavelet}
	We say that an infinite matrix $M\in\R^{\N\times\N}$ is of Jaffard class, $M\in \JJ(d,\gamma,C)$ for
	some metric $d(\cdot,\cdot)\colon \N\times \N\to [0,\infty)$ and some $\gamma>0$ if for all $0<\gamma^\prime<\gamma$ there exists $C({\gamma^\prime})>0$ such that
	\begin{align}\label{eq:Mexp}
	|M_{ij}|\leq C({\gamma^\prime}) \exp(-\gamma^\prime d(i,j))\quad\text{for all }i,j\in\N.
	\end{align}
	Moreover, the metric $d(\cdot,\cdot)$ must satisfy for all $\eps>0$ 
	\begin{align}\label{eq:Mdist}
	\sup_{i\in\N}\sum_{j\in\N}\exp(-\eps d(i,j))<\infty.
	\end{align}
	We also write $M\in \JJ$ to state the existence of parameters $d,\gamma,C$ such that $M\in \JJ(d,\gamma,C)$.
\end{definition}

\begin{definition}[banded matrix]\label{def:banded}
	We say that an infinite matrix $M\in\R^{\N\times\N}$ is banded with respect to some metric $d(\cdot,\cdot)\colon \N\times \N\to [0,\infty)$ if there exists a bandwidth $b\geq 1$ such that
	\begin{align}\label{eq:banded}
	d(i,j)> b \quad\implies\quad M_{ij}=0\quad\text{for all }i,j\in\N.
	\end{align}
	In this case, we write $M\in \BB(d,b)$.
	Note that we do not require $d(\cdot,\cdot)$ to satisfy~\eqref{eq:Mdist}.
	We also write $M\in \BB$ or $M\in\BB(d)$ to state that the missing parameters exist.
\end{definition}	 

The following technical lemmas state some straightforward facts about infinite matrices and can be found in~\cite{fembemopt}.
\begin{lemma}\label{lem:bandmult}
	Let $M^{i,j}\in \BB(d,b_{j})$, $i=1,\ldots,n$, $j=1,\ldots,m$ for some $m,n\in\N$ with respect to some metric $d(\cdot,\cdot)$ and respective bandwidths $b_{j}\in\N$.
	Then, there holds 
	\begin{align*}
	\sum_{i=1}^n \prod_{j=1}^m M^{i,j}\in\BB(d,\sum_{j=1}^m b_j). 
	\end{align*}
\end{lemma}

\begin{lemma}\label{lem:norm}
	Let $M\in\JJ(d,\gamma,C)$, then $|M|\colon \ell_2\to\ell_2$ is a bounded operator
	(the modulus $|M|$ is understood entry wise).
\end{lemma}

\begin{lemma}\label{lem:invbound}
	Let $M\in\JJ(d,\gamma,C)$ and assume a block structure $n_1,n_2,\ldots\in\N$ such that
	$M$ satisfies~\eqref{eq:matinfsup}.
	Then, 
	\begin{align*}
	\overline{M}\in\R^{\N\times\N}\quad\text{with}\quad	\overline{M}_{ij}:= \sup_{\stacktwo{k\in\N}{ n_k\geq \max\{i,j\}}}|(M[k])^{-1}|_{ij}
	\end{align*}
	is in $\JJ(d,\widetilde \gamma,\widetilde C)$ and thus a bounded operator $\overline{M}\colon \ell_2\to\ell_2$.
	The constant $\widetilde\gamma$ depends only on $C_{\rm inf-sup}>0$, $d$, and $\gamma$, whereas for all $0<\gamma^\prime< \widetilde \gamma$, $\widetilde C(\gamma^\prime)$ depends only on an upper bound for $C(\gamma^\prime)$
	and on $C_{\rm inf-sup}>0$.
\end{lemma}
\begin{proof} 
	The result~\cite[Proposition~2]{rescue}
	shows that $(M[k])^{-1}\in\JJ(d,\widetilde \gamma,\widetilde C)$. Inspection of the proof reveals that
	$\widetilde\gamma$ depends only on $\gamma$, $d$, and $\widetilde C(\gamma^\prime)$ depends only on an upper bound for $C(\gamma^\prime)$ from Definition~\ref{def:wavelet}
	and on $C_{\rm inf-sup}>0$. Therefore, we have for all $0<\gamma^\prime<\widetilde \gamma$
	\begin{align*}
	|\overline{M}_{ij}| \leq \widetilde C(\gamma^\prime)\exp(-\gamma^\prime d(i,j))\quad\text{for all }i,j\in\N
	\end{align*}
	and hence $\overline{M}\in\JJ(d,\widetilde\gamma,\widetilde C)$. Lemma~\ref{lem:norm} concludes the proof.
\end{proof}

\subsubsection{LU-factorization}\label{section:lu}

The following results show existence of a bounded $LU$-factorization for particular Jaffard class matrices.
\begin{theorem}\label{thm:LU}
	Let $M\in\R^{\N\times \N}\in\JJ(d,\gamma,C)$ and additionally satisfy~\eqref{eq:matinfsup} for some given block-structure
	in the sense of Definition~\ref{def:blockstruct}.
	Then, $M$ has a block-$LU$-factorization such that $L,U,L^{-1},U^{-1}\colon \ell_2\to\ell_2$ are bounded operators with operator norms depending only
	on $\JJ(d,\gamma,C)$, $C_{\rm inf-sup}$, and $\norm{M}{2}$.
\end{theorem}
\begin{proof}
    We repeat the proof of~\cite[Theorem~2]{LU} to show that it also works for block-LU-factorization. The existence of block-lower/block-upper triangular matrices $L,U\in\R^{\N\times\N}$ such that
    $M=LU$ is well-known. The identity~\eqref{eq:luid} shows for $x\in\ell_2$ and $\overline{M}$ from Lemma~\ref{lem:invbound}
    \begin{align*}
     \norm{U^{-1}x}{\ell_2}^2&\leq \sum_{i=1}^\infty\norm{\sum_{k=i}^\infty|M[k]^{-1}(i,k)||x|_{\{n_k,\ldots,n_{k+1}-1\}}|}{\ell_2}^2\\
     &\leq \sum_{i=1}^\infty\norm{\sum_{k=i}^\infty \overline{M}(i,k)|x|_{\{n_k,\ldots,n_{k+1}-1\}}|}{\ell_2}^2\leq \norm{\overline{M}}{2}\norm{|x|}{\ell_2}\lesssim\norm{x}{\ell_2}.
    \end{align*}
    This shows that $\norm{U^{-1}}{2}<\infty$ and we deduce immediately $\norm{L}{2}\leq \norm{M}{2}\norm{U^{-1}}{2}<\infty$.
    Lemma~\ref{lem:blockLU0} shows additionally that the block-diagonal matrix $D(i,i):= U(i,i)$, $i\in\N$ is bounded $\norm{D}{2}<\infty$ in terms of $C_{\rm inf-sup}$ and $\norm{M}{2}$.
    The same argumentation proves for $M^T=\widetilde L\widetilde U$ that $\norm{\widetilde L}{2}+\norm{\widetilde U^{-1}}{2}<\infty$.
    From~\eqref{eq:lut}, we see
    \begin{align*}
     \norm{U}{2}=\norm{\widetilde L}{2}\norm{D^T}{2}<\infty\quad\text{and hence}\quad \norm{L^{-1}}{2}\leq \norm{U}{2}\norm{M^{-1}}{2}<\infty.
    \end{align*}
This concludes the proof.
\end{proof}

Below we present the main theorem of this section. It proves that the $LU$-factors of a banded matrix can be approximated by banded matrices. In some sense, the result is an extension of~\cite{LUpaper} to a particular case.

\begin{theorem}\label{thm:blockLU}
	Let $M\in\BB(d,b_0)$ for some $b_0\in\N$ and some metric $d(\cdot,\cdot)$ such that $M\colon\ell_2\to \ell_2$ is bounded and satisfies~\eqref{eq:matinfsup}
	for some block structure in the sense of Definition~\ref{def:blockstruct}. Let $M$
	be block-banded in the sense $M(i,j)=0$ for $|i-j|>b_0$. Then, for all $\eps>0$,	
	there exists an approximate block-$LDU$-decomposition $\norm{M-LDU}{2}\leq \eps$ for block-lower/block-upper triangular factors $L$, $U$ such that $L(i,i)=U(i,i)=I$ for all $i\in\N$,
	and a block diagonal factor $D$. 
	The factors $L,D,U\colon \ell_2\to \ell_2$ are bounded with bounded inverses uniformly in $\eps$ and satisfy $L^{-1},D,U^{-1}\in \BB(d,b)$ for some bandwidth $b$.
	Moreover, $L^{-1},U^{-1}$ are block-banded with bandwidth $b$, i.e., $L^{-1}(i,j)=U^{-1}(i,j)=0$ for all $|i-j|>b$. Additionally, $D$ satisfies~\eqref{eq:matinfsup}.
	The constant $b$ depends only on $M$, $b_0$, and $\eps$.
\end{theorem}
\begin{remark}
 The main idea behind the proof of Theorem~\ref{thm:blockLU} is to exploit the well-known property of the $LU$-factorization to inherit skyline structure and particularly (block-) bandedness of the parent matrix. The theorem goes one step further and proves that the combination of block-bandedness and bandedness in the sense of Definition~\ref{def:banded} can be preserved approximately. We use this fact to show that the approximate $LU$-factors are inversely bounded (which seems to be inaccessible without any additional structure).          
\end{remark}

We postpone the proof of Theorem~\ref{thm:blockLU} to the end of this section.

\begin{lemma}\label{lem:invapprox}
	Let $M\in\R^{\N\times\N}$ such that $M\colon\ell_2\to \ell_2$ is bounded and satisfies~\eqref{eq:matinfsup} for
	some block structure in the sense of Definition~\ref{def:blockstruct}. Moreover, let $M\in \BB(d,b_0)$.
	Then, given $\eps>0$, there exists a bandwidth $b\in\N$ such that for all $k\in\N$, there exist $R,R_k\in \BB(d,b)$ such that
	\begin{align*}
	\norm{M^{-1}-R}{2}+\sup_{k\in\N} \norm{M[k]^{-1}-R_k}{2}\leq \eps.
	\end{align*}
	If $M$ is additionally block-banded  in the sense $M(i,j)=0$ for all $|i-j|>b_0$, then, $R_k$ and $R$ will additionally be block-banded with bandwidth $b$.
	If $M$ is block-diagonal, also $R$ and $R_k$ will be block-diagonal.
	The bandwidth $b$ depends only on $b_0$, $C_{\rm inf-sup}$, $\norm{M}{2}$, and $\eps$.
\end{lemma}
\begin{proof}
	Let $A:=M[k]^TM[k]$ or $A:=M^TM$. Due to~\eqref{eq:matinfsup}, $A$ is elliptic with some constant $C_{\rm inf-sup}^{-2}$. We obtain for $\alpha:=C_{\rm inf-sup}^{-2}\norm{M}{2}^{-4}$ and $x\in\ell_2$
	\begin{align*}
	\norm{x-\alpha Ax}{\ell_2}^2 &= \norm{x}{\ell_2}^2 - 2\alpha\dual{x}{Ax}_{\ell_2} + \alpha^2 \norm{Ax}{\ell_2}^2\\
	&\leq (1-2\alpha C_{\rm inf-sup}^{-2}+\alpha^2\norm{A}{2}^2)\norm{x}{\ell_2}^2\leq(1-C_{\rm inf-sup}^{-4}/\norm{M}{2}^4)\norm{x}{\ell_2}^2.
	\end{align*}
	This shows $\norm{I-\alpha A}{2}\leq (1-C_{\rm inf-sup}^{-4}/\norm{M}{2}^4):= q<1$.
	We obtain
	\begin{align*}
	A^{-1} = \alpha(\alpha A)^{-1} = \alpha(I-(I-\alpha A))^{-1} =\alpha\sum_{j=0}^\infty (I-\alpha A)^{j}.
	\end{align*}
	We define $R:= (\sum_{j=0}^N (I-\alpha A)^j)M^T$ and $R_k:=(\sum_{j=0}^N (I-\alpha A)^j)M[k]^T$ for some $N\in\N$ such that $\norm{\sum_{j=N+1}^\infty (I-\alpha A)^j}{2}\leq \sum_{j=N+1}^\infty q^j\leq \eps/\norm{M}{2}$.
	Since $(I-\alpha A)\in\BB(d,b_0)$ and $M\in \BB(d,b_0)$, Lemma~\ref{lem:bandmult} shows that $R$ and $R_k$ are banded as well.
	The bandwidth depends only on $b_0$, $q$, and $N$.
	If $M$ is additionally block-banded, also $A$ and $(I-\alpha A)$ will be block-banded with bandwidth $2b_0$. Hence $(I-\alpha A)^k$ will be block-banded with bandwidth $2kb_0$. 
	The same argumentation proves the statement for block-diagonal $M$.
	This concludes the proof.
\end{proof}

The following results prove that block-banded matrices $M$ hand down some structure to their $LU$-factors. 
This section is similar to~\cite[Section~3.2]{fembemopt}, however, with the difference that we include indefinite matrices instead
of elliptic ones.
\begin{lemma}\label{lem:blockLU0}
 Let $M\in\R^{\N\times\N}$ such that $M\colon\ell_2\to \ell_2$ is bounded and satisfies~\eqref{eq:matinfsup} for a
	 block structure in the sense of Definition~\ref{def:blockstruct}.
	Then, the block-$LU$-factorization $M=LU$ for block-lower/block-upper triangular matrices $L,U\in\R^{\N\times\N}$ such that $L(i,i)=I$ for all $i\in\N$ exists and satisfies
	\begin{align*}
	 \sup_{k\in\N}\norm{U(k,k)}{2}\leq (1+\norm{M}{2})^2C_{\rm inf-sup}.
	\end{align*}
\end{lemma}
\begin{proof}
It is well-known that the block-$LU$-factorization exists. We further note that $M[k]=L[k]U[k]$ for all $k\in\N$, i.e., restriction to principal submatrices commutes with the block-LU-factorization.
 Therefore, to see that $U(k,k)$ is uniformly bounded, we may restrict to $M[k]$ for $k\in\N$.
 For matrices $R_1,R_2\in \R^{(n_k-1)\times (n_{k+1}-n_{k})}$ and $R_3\in \R^{(n_{k+1}-n_{k})\times (n_{k+1}-n_{k})}$ the $(2\times 2)$-block-$LU$-factorization reads
	\begin{align*}
	 M[k]&=\begin{pmatrix}
	 M[k-1] & R_1\\ R_2^T & R_3
	 \end{pmatrix}\\
 &	 =
	 \begin{pmatrix}
	 I & 0\\ R_2^TM[k-1]^{-1} & I
	 \end{pmatrix}
	 \begin{pmatrix}
	 M[k-1] & R_1\\0 & R_3-R_2^TM[k-1]^{-1}R_1
	 \end{pmatrix}.
	\end{align*}
	Uniqueness of normalized block-$LU$-factorization (further factorization of $M[k-1]$ will not alter the lower-right block of the $U$-factor) implies $U(k,k)= R_3-R_2^TM[k-1]^{-1}R_1$ and hence
	$\norm{U(k,k)}{2}\leq \norm{M[k]}{2}+\norm{M[k]}{2}^2C_{\rm inf-sup}\leq (1+\norm{M}{2})^2C_{\rm inf-sup}$ , where we used that the norm of the sub matrices $R_1,R_2,R_3$
	is bounded by the norm of the matrix $M[k]$ as well as $\norm{M[k-1]^{-1}}{2}\leq C_{\rm inf-sup}$.
\end{proof}

\begin{lemma}\label{lem:blockLU1}
	Let $M\in\R^{\N\times\N}$ such that $M\colon\ell_2\to \ell_2$ is bounded and satisfies~\eqref{eq:matinfsup} for some
	 block structure in the sense of Definition~\ref{def:banded}.
	Moreover, let $M$ be block-banded in the sense $M(i,j)=0$ for $|i-j|>b_0$ for some $b_0\in\N$.
	Then, the block-$LU$-factorization $M=LU$ for block-lower/block-upper triangular matrices $L,U\in\R^{\N\times\N}$ such that $L(i,i)=I$ for all $i\in\N$ exists, is block-banded with bandwidth $b_0$,
	and satisfies
	\begin{align*}
	\norm{L}{2}+\norm{U}{2}+\norm{L^{-1}}{2}+\norm{U^{-1}}{2}<\infty.
	\end{align*}
	Moreover, the block-diagonal matrix $D\in\R^{\N\times\N}$, $D(i,i):= U(i,i)$ as well as its inverse are bounded and satisfy~\eqref{eq:matinfsup}.
\end{lemma}
\begin{proof}
	Since $M[k]$ is invertible for all $k\in\N$, it is a well-known fact that the block-LU-factorization exists. 
	Since $M[k]$ and $L[k]$ are invertible by definition, 
	$U[k]$ is invertible for all $k\in\N$. The block-triangular structure guarantees that $L[k]^{-1}=(L^{-1})[k]$ as well as $U[k]^{-1}=(U^{-1})[k]$, and hence existence of $L^{-1},U^{-1}$ as matrices in $\R^{\N\times\N}$.
	Moreover, it is well-known that $L$ and $U$ are block-banded with bandwidth $b_0$.
	By definition, we have $M[k]^{-1} = U[k]^{-1} L[k]^{-1}$. Since $L$ is lower-triangular with normalized block-diagonal (only identities in the diagonal blocks), the same is true for $L[k]^{-1}$. Therefore, we obtain
	\begin{align}\label{eq:luid}
	M[k]^{-1}(i,k) = \sum_{r=1}^kU[k]^{-1}(i,r) L[k]^{-1}(r,k) =U[k]^{-1}(i,k)=U^{-1}(i,k),
	\end{align}
	where the last identity follows from the fact that $U^{-1}$ is upper-block triangular.
	
	We see that 
	$\sup_{i,j}\norm{U^{-1}(i,j)}{2}\leq\sup_{k\in\N}\norm{M[k]^{-1}}{2}\leq C_{\rm inf-sup}<\infty$. Hence, there holds that
	$\norm{D(k,k)^{-1}}{2}=\norm{U^{-1}(k,k)}{2}\leq C_{\rm inf-sup}$. Moreover, we obtain
	\begin{align*}
	\sup_{i,j\in\N}\norm{L(i,j)}{2}\leq \sup_{i,j\in\N}\sum_{\stacktwo{k\in\N_0}{|k-i|\leq b_0}} \norm{M(i,k)U^{-1}(k,j)}{2}\leq 2b_0\norm{M}{2}\sup_{i,j}\norm{U^{-1}(i,j)}{2}<\infty.
	\end{align*}
	Since $L$ is block-banded with bandwidth $b_0$, the result~\cite[Lemma~8.4]{fembemopt} shows $\norm{L}{2}<\infty$. This implies $\norm{U^{-1}}{2}\leq \norm{M^{-1}}{2}\norm{L}{2}<\infty$.
	Lemma~\ref{lem:blockLU0} shows that $D(k,k)=U(k,k)$ is uniformly bounded.
	Thus, we proved $\norm{D}{2}+\norm{D^{-1}}{2}<\infty$ depending only on $C_{\rm inf-sup}$ and $\norm{M}{2}$.

	Let $M^T =\widetilde L\widetilde U$ be the analogous block-LU-factorization for the transposed matrix (note that $M^T$ still satisfies~\eqref{eq:matinfsup} and is bounded and banded).
	Since normalized $LU$-factorizations are unique, we see that
	\begin{align}\label{eq:lut}
	\widetilde L = U^TD^{-T}\quad \text{and}\quad \widetilde U= D^TL^T.
	\end{align}
	Repeating the above arguments shows $\norm{\widetilde L}{2}+\norm{\widetilde U^{-1}}{2}<\infty$. With boundedness of $D$ and $D^{-1}$,~\eqref{eq:lut} shows $\norm{U}{2}+\norm{L^{-1}}{2}<\infty$ and hence
	concludes the proof.
\end{proof}

\begin{lemma}\label{lem:blockLU2}
	Under the assumptions of Lemma~\ref{lem:blockLU1}, assume that additionally $M\in \BB(d,b_0)$.
	Given $\eps>0$, there exists $b\in\N$ and block-upper triangular $U_\eps^{-1}\in\BB(d,b)$ which is additionally block-banded in the sense $U_\eps^{-1}(i,j)=0$ for $|i-j|>b$ such that
	\begin{align*}
	\norm{U^{-1}-U_\eps^{-1}}{2}\leq \eps.
	\end{align*}
	The approximation $U_\eps^{-1}$ is invertible with bounded inverse such that $\sup_{\eps>0}(\norm{U_\eps}{2}+\norm{U_\eps^{-1}}{2})<\infty$.
	Moreover, there exists block-diagonal $D_\eps\in\BB(d,b)$ which is bounded and satisfies~\eqref{eq:matinfsup} such that $\norm{D-D_\eps}{2}\leq \eps$.
\end{lemma}
\begin{proof}
	Lemma~\ref{lem:invapprox} shows that there exist $R,R_k\in \BB(d,b)$ which are block-banded with bandwidth $b=b(\eps)$ such that $\norm{M^{-1}-R}{2}+\norm{M[k]^{-1}-R_k}{2}\leq \eps$ for all $k\in\N$.
	Choosing $\eps>0$ sufficiently small, 
	we ensure that also $R$ and $R_k$ are bounded and satisfy~\eqref{eq:matinfsup} with uniform
	constant. 
	
	Inspired by~\eqref{eq:luid}, we define a first approximation to $U^{-1}$ by 
	\begin{align}\label{eq:defueps}
	T(i,j):=\begin{cases}
	0 &\text{for }j<i\text{ or }i<j-b,\\
	R_j(i,j) &\text{for }j-b\leq i\leq j.
	\end{cases}
	\end{align}
	This ensures that $T\in \BB(d,b)$ and  that $T$ is block-banded with bandwidth $b$.
	Additionally, we obtain
	\begin{align}\label{eq:err1}
	\sup_{i,j\in\N}\norm{T(i,j)-U^{-1}(i,j)}{2}\leq \eps.
	\end{align}	
	We define an approximation to $L$ (which is block-banded with bandwidth $b_0$) by 
	\begin{align*}
	S(i,j):=\begin{cases}
	0 & \text{for }j<i-b_0\text{ or } j>i,\\
	I &\text{for }j=i,\\
	(M T)(i,j)&\text{for } i-b_0\leq j<i.
	\end{cases}
	\end{align*}
	The definition and~\eqref{eq:err1} imply 
	\begin{align}\label{eq:err2}
	\norm{L(i,j)-S(i,j)}{2}&\leq \norm{\sum_{k=i-b_0}^{i+b_0} M(i,k)(U^{-1}(k,j)-T(k,j))}{2}\\
	&\leq \sum_{k=i-b_0}^{i+b_0}\norm{M}{2}\norm{U^{-1}(k,j)-T(k,j))}{2}\lesssim \norm{M}{2}b_0 \eps.
	\end{align}
	Since both $L$ and $S$ are block-banded with bandwidth $b_0$, the result~\cite[Lemma~8.4]{fembemopt} shows even
	\begin{align}\label{eq:err3}
	\norm{L-S}{2}\lesssim \eps,
	\end{align}
	where the hidden constant is independent of $\eps$. Moreover, Lemma~\ref{lem:bandmult} shows that $S\in \BB(d,\widetilde b)$,
	for some $\widetilde b\in\N$ which depends only on $b_0$ and $b$.
	
	Recall $R$ from above with $\norm{M^{-1}-R}{2}\leq \eps$, $R\in\BB(d,b)$ and $R$ is block-banded with bandwidth $b$.
	This allows us to define $U_\eps^{-1}$ by
	\begin{align*}
	U_\eps^{-1}:= RS.
	\end{align*}
	We obtain from the definition and with~\eqref{eq:err3}
	\begin{align}\label{eq:err4}
	\begin{split}
	\norm{U_\eps^{-1}-U^{-1}}{2}&\leq \norm{R(S-L)}{2}+\norm{(R-M^{-1})L}{2}\lesssim\norm{R}{2}\eps +\norm{L}{2}\eps\\
	&\leq (C_{\rm inf-sup}+\eps+1)\eps,
	\end{split}
	\end{align}
	where the hidden constant does not depend on $\eps$.
	Moreover, Lemma~\ref{lem:bandmult} shows (since $S$ and $R$ are block-banded), that $U_\eps^{-1}\in\BB(d)$ with bandwidth depending on $\eps$.
	Analogously, we see $U_\eps^{-1}$ is block-banded with bandwidth $b_0+b$. Since $U^{-1}$ is invertible with bounded inverse, choosing $\eps>0$ sufficiently small ensures
	that $U_\eps^{-1}$ is invertible, with bounded inverse uniformly in $\eps$.
	
	Let $\widetilde D$ denote the block-diagonal of $U_\eps^{-1}$.
	Obviously, $\widetilde D\in\BB(d)$ and~\eqref{eq:err4} implies $\norm{\widetilde D-D^{-1}}{2}\lesssim \eps$. Lemma~\ref{lem:blockLU1} shows that $D$ and $D^{-1}$ are
	bounded, thus sufficiently small 
	$\eps>0$ guarantees the same for $\widetilde D$ and $\widetilde D^{-1}$. Since $\widetilde D$ is block-diagonal, this implies~\eqref{eq:matinfsup} for $\widetilde D$.
	Hence, Lemma~\ref{lem:invapprox} ensures that there exists block-diagonal $D_\eps\in\BB(d)$ (with bandwidth depending only on $\eps>0$), such that
	$\norm{\widetilde D^{-1}-D_\eps}{2}\leq \eps$. From this, we obtain
	\begin{align*}
	\norm{D-D_\eps}{2}&\leq \norm{\widetilde D^{-1}-D_\eps}{2}+\norm{\widetilde D^{-1}-D}{2}\leq \eps +\norm{\widetilde D^{-1}}{2}\norm{D}{2}\norm{\widetilde D-D^{-1}}{2}\\
	&\lesssim (1 +\norm{\widetilde D^{-1}}{2}\norm{D}{2})\eps.
	\end{align*}
	For sufficiently small $\eps>0$, $\norm{\widetilde D^{-1}}{2}$ is bounded in terms of $\norm{D}{2}$. This ensures that the constant above does not depend on $\eps$ and thus concludes the proof.
\end{proof}

\begin{lemma}\label{lem:blockLU3}
	Under the assumptions of Lemma~\ref{lem:blockLU1}--\ref{lem:blockLU2},
	there exists $b\in\N$ and block-lower triangular $L_\eps^{-1}\in\BB(d,b)$ with $L_\eps^{-1}(i,i)=I$, which is additionally block-banded in the sense $L_\eps^{-1}(i,j)=0$ for $|i-j|>b$ such that
	\begin{align*}
	\norm{L^{-1}-L_\eps^{-1}}{2}\leq \eps.
	\end{align*}
	The approximation $L_\eps^{-1}$ is invertible such that $\sup_{\eps>0}(\norm{L_\eps}{2}+\norm{L_\eps^{-1}}{2})<\infty$.
\end{lemma}
\begin{proof}
	Recall that $M^T$ satisfies all the assumptions of Lemma~\ref{lem:blockLU1}--\ref{lem:blockLU2}.
	Let $M^T=\widetilde L \widetilde U$. We apply Lemma~\ref{lem:blockLU2} to $M^T$ to obtain an approximation $\widetilde U_\eps^{-1}\in\BB(d,b)$, 
	block-banded with bandwidth $b$, bounded with bounded inverse (uniformly in $\eps$) such that
	\begin{align*}
	\norm{\widetilde U^{-1}-\widetilde U_\eps^{-1}}{2}\leq \eps.
	\end{align*}
	The identity~\eqref{eq:lut} shows $L^{-1}=D\widetilde U^{-T}$ and thus motivates the definition
	\begin{align*}
	L_\eps^{-1}(i,j):=\begin{cases} (D_\eps\widetilde U_\eps^{-T})(i,j)& i\neq j,\\
	I &i=j.
	\end{cases}
	\end{align*}
	with $D_\eps\in\BB(d,b)$ from Lemma~\ref{lem:blockLU2} applied to $M$. Lemma~\ref{lem:bandmult} shows $L_\eps^{-1}\in \BB(d)$ and $L_\eps^{-1}$ is also block-banded with bandwidth $b$.
	We obtain with the approximation estimates from Lemma~\ref{lem:blockLU1}
	\begin{align*}
	\norm{L_\eps^{-1}-L^{-1}}{2}&\leq \norm{D_\eps\widetilde U_\eps^{-T} - D\widetilde U^{-T}}{2} + \sup_{i\in\N}\norm{(D_\eps\widetilde U_\eps^{-T})(i,i)-I}{2}
	\end{align*}
	
	The first term on the right-hand side is bounded by
	\begin{align*}
	\norm{D_\eps\widetilde U_\eps^{-T} - D\widetilde U^{-T}}{2}&\leq \norm{(D_\eps-D)\widetilde U_\eps^{-T}}{2}+\norm{D(\widetilde U_\eps^{-T}-\widetilde U^{-T})}{2}\\
	&\leq 
	\eps(\norm{\widetilde U^{-1}}{2}+\eps)+\norm{D}{2}\eps\lesssim \eps.
	\end{align*}
	The second term satisfies
	\begin{align*}
	\sup_{i\in\N}\norm{(D_\eps\widetilde U_\eps^{-T})(i,i)-I}{2}\leq \norm{D_\eps\widetilde U_\eps^{-T}-L}{2}=\norm{D_\eps\widetilde U_\eps^{-T} - D\widetilde U^{-T}}{2}\lesssim \eps.
	\end{align*}
	Choosing $\eps>0$ sufficiently small ensures that $L_\eps^{-1}$ is invertible with bounded inverse uniformly in $\eps>0$. This concludes the proof.
\end{proof}

We are now in the position to prove the main theorem of this section.

\begin{proof}[Proof of Theorem~\ref{thm:blockLU}]
	To avoid confusion, we denote the $LU$-factorization of $M$ from Lemma~\ref{lem:blockLU1} by $\widetilde L$ and $\widetilde U$, with diagonal matrix $\widetilde D$.
	With Lemma~\ref{lem:blockLU2}--\ref{lem:blockLU3}, we set $D:=D_\eps$ and $L^{-1}:= L_\eps^{-1}$. This ensures $L^{-1},D\in\BB(d)$ and that $L^{-1}$ is block-banded. Moreover, $D$ is
	bounded and satisfies~\eqref{eq:matinfsup}.
	This motivates the definition 
	\begin{align*}
	U^{-1}(i,j):=\begin{cases} ( U_\eps^{-1}D_\eps)(i,j)&i\neq j,\\
	I&i=j.
	\end{cases}
	\end{align*}
	Lemma~\ref{lem:bandmult} shows that $U^{-1}\in\BB(d)$ with bandwidth depending on $\eps$ and moreover $U^{-1}$ is block-banded.
	We obtain
	\begin{align*}
	\norm{M^{-1}-U^{-1}D^{-1}L^{-1}}{2}&\leq \norm{M^{-1}-U_\eps^{-1}L_\eps^{-1}}{2} + \norm{(U_\eps^{-1}-U^{-1}D_\eps^{-1})L_\eps^{-1}}{2}\\
	&\leq \norm{M^{-1}-U_\eps^{-1}L_\eps^{-1}}{2} +\sup_{i\in\N}\norm{ ( U_\eps^{-1}D_\eps)(i,i)-I}{2}\norm{L_\eps^{-1}}{2}\\
	&\leq \norm{M^{-1}-U_\eps^{-1}L_\eps^{-1}}{2}+\norm{U_\eps^{-1}D_\eps -\widetilde U^{-1} \widetilde D}{2}\norm{L_\eps^{-1}}{2}.
	\end{align*}
	The first term on the right-hand side can be bounded by use of Lemma~\ref{lem:blockLU2}--\ref{lem:blockLU3} by
	\begin{align*}
	\norm{M^{-1}-U_\eps^{-1}L_\eps^{-1}}{2}\leq \norm{\widetilde U^{-1}}{2}\norm{\widetilde L^{-1}-L_\eps^{-1}}{2}+ \norm{\widetilde U^{-1}-U_\eps^{-1}}{2}\norm{L_\eps^{-1}}{2}\lesssim \eps,
	\end{align*}
	where the hidden constant does not depend on $\eps>0$.
	The second term can be bounded in a similar fashion by
	\begin{align*}
	\norm{U_\eps^{-1}D_\eps &-\widetilde U^{-1} \widetilde D}{2}\norm{L_\eps^{-1}}{2}\\
	&\leq \norm{U_\eps^{-1}}{2}\norm{D_\eps - \widetilde D}{2}\norm{L_\eps^{-1}}{2}+\norm{U_\eps^{-1} -\widetilde U^{-1}}{2}\norm{\widetilde D}{2}\norm{L_\eps^{-1}}{2}\lesssim \eps
	\end{align*}
	with $\eps$-independent hidden constant.
	Altogether we proved
	\begin{align*}
	\norm{M^{-1}-U^{-1}D^{-1}L^{-1}}{2}\lesssim \eps.
	\end{align*}
	Moreover, there holds
	\begin{align*}
	\norm{M-LDU}{2}&\leq \norm{M}{2}\norm{LDU}{2} \norm{M^{-1}-U^{-1}D^{-1}L^{-1}}{2}\\
	&\leq\norm{M}{2}(\norm{M}{2}+\norm{LDU-M}{2}) \norm{M^{-1}-U^{-1}D^{-1}L^{-1}}{2}\\
	&  \lesssim   \norm{M}{2}^2\eps
	+ \norm{M}{2}\norm{M-LDU}{2} \eps.
	\end{align*}
	Sufficiently small $\eps>0$ shows
	\begin{align*}
	\norm{M-LDU}{2}\lesssim \eps,
	\end{align*}
	where the hidden constant does not depend on $\eps$. This concludes the proof.
\end{proof}

\section{Riesz bases}\label{section:basis}
This section constructs suitable Riesz bases of $H^1_0(\Omega)$ and $L^2_\star(\Omega)$ for the velocity in Theorem~\ref{thm:stableB1} and for the pressure in Theorem~\ref{thm:stableB0} of the Stokes problem.
To that end, we define an auxiliary sequence of uniform meshes.
\begin{definition}\label{def:widehat}
	We consider an auxiliary sequence $(\widehat\TT_\ell)_{\ell\in\N}$ of uniform refinements such that $\widehat\TT_0=\TT_0$ and
	\begin{align*}
		\widehat\TT_{\ell+1}={\rm bisec5}^k(\widehat\TT_\ell,\widehat\TT_\ell),
	\end{align*}
	which means that each element of $\widehat\TT_\ell$ is refined $k$-times with {\rm bisec5} to obtain $\widehat\TT_{\ell+1}$.
	There exist constants $C_{\rm base},C_{\rm mesh}> 1$ which depend on $k$ and  on $\TT_0$ such that
	\begin{align*}
	C^{-1}_{\rm base}C_{\rm mesh}^{-\ell}\leq	{\rm diam}(T)\leq C_{\rm base}C_{\rm mesh}^{-\ell}
	\end{align*}
	for all $T\in\widehat\TT_\ell$ and all $\ell\in\N$.
	We choose $k$ sufficiently large such that $C_{\rm mesh}\geq (C_{\rm sz}+1)^4$, where $C_{\rm sz}$ is defined in Lemmas~\ref{lem:sz1}\&\ref{lem:sz2} below.	
\end{definition}
Moreover, we use the fact that $(H^{3/4}(\Omega),H^1(\Omega),H^{5/4}(\Omega))$ form a Gelfand triple with $H^1(\Omega)$ as its pivot space as is shown in the following lemma.

\begin{lemma}[from~{\cite{fembemopt}}]\label{lem:gelfand}
For $0<s<1/2$, the interpolation spaces $H^{1-s}(\Omega)$ and $H^{1+s}(\Omega)$ form a Gelfand triple in the sense $H^{1+s}(\Omega)\subset H^1(\Omega)\subset H^{1-s}(\Omega)$. 
This means that functionals of the form $H^{1+s}(\Omega)\to \R,\,v\mapsto (v,w)_{H^1(\Omega)}$, $w\in H^{1}(\Omega)$ can be identified with a dense subspace of $H^{1-s}(\Omega)$.
\end{lemma}

To construct the Riesz basis, we define a natural hierarchical basis on the spaces $\SS^p(\TT)$.
\begin{definition}\label{def:s2}
	Given a triangulation $\TT$, define the hat functions $v_z\in \SS^1(\TT)$ associated with a certain node $z$ of $\TT$. 
	For an edge $E$ of $\TT$ with endpoints $z_1$ and $z_2$, define the edge bubble $v_E:= \alpha_E v_{z_1}v_{z_2}$ with $\alpha_E>0$ such that $\norm{v_E}{L^\infty(\Omega)}=1$.
	Let $\SS^{p}_B(\TT)$ for $p\in\{1,2\}$ denote the set of all hat-, resp. hat and edge bubble-functions defined on $\TT$ and let $\SS^{p}(\TT)$ define
	the linear span of $\SS^{p}_B(\TT)$.
	For a refinement $\widehat\TT$ of $\TT$, we denote by $\SS^{p}_B(\widehat\TT\setminus\TT)$ all hat functions $v_z$ associated with new nodes 
	$z\in\NN(\widehat\TT)\setminus\NN(\TT)$ and (for $p=2$) also
	all edge bubble functions $v_E$ associated with new edges $E\in\EE(\widehat\TT)\setminus \EE(\TT)$. By $\SS^p_{B,0}(\TT)$ and $\SS^p_{B,0}(\widehat\TT\setminus\TT)$, 
	we denote the bases which
	vanish on $\partial\Omega$.
\end{definition}

The following theorem establishes the Riesz basis for $H^1_0(\Omega)$ which contains the velocity variable of the Stokes problem.
\begin{definition}\label{def:b1tilde}
 We define the basis $\widetilde B_0^1$ as an arbitrary basis of $\SS^2_0(\widehat\TT_0)$. For $\ell\geq 1$, define
 \begin{align*}
  \widetilde B_\ell^1:=\set{\frac{v}{\norm{v}{H^1(\Omega)}}}{ v\in
  \Big(\bigcup_{k\in\N}\SS^{2}_{B,0}(\TT_k\setminus\TT_{k-1})\Big)\cap (\SS^{2}(\widehat\TT_\ell)\setminus
\SS^{2}(\widehat\TT_{\ell-1}))},
 \end{align*}
\end{definition}

To prove the following theorem, we aim to employ the multiscale decomposition result by Dahmen~\cite{Dahmen}. To that end, we require projection operators
$S^2_\ell\colon H^1(\Omega)\to \SS^2(\widehat\TT_\ell)$ which are uniformly $H^1(\Omega)$-bounded and satisfy $S^2_\ell S^2_k= S^2_\ell$ for all $\ell\leq k$.
Moreover, the following approximation estimates
\begin{align*}
\norm{(1-S^2_\ell) u}{H^{3/4}(\Omega)}&\lesssim C_{\rm mesh}^{-\ell/4}\norm{u}{H^1(\Omega)},\\
\norm{(1-S^2_\ell) u}{H^{1}(\Omega)}&\lesssim C_{\rm mesh}^{-\ell/4}\norm{u}{H^{5/4}(\Omega)}
\end{align*}
need to hold as well as uniform boundedness $S^2_\ell\colon H^{3/4}(\Omega)\to H^{3/4}(\Omega)$. An obvious choice for such an operator would be the 
Scott-Zhang projection from~\cite{scottzhang} (see Lemma~\ref{lem:sz2} for details). However, the standard operator does not satisfy the required symmetry $S_\ell^2S_k^2 = S_\ell^2$ for $\ell\leq k$. Therefore,
we introduce a modified Scott-Zhang operator $S_\ell^2$ in Theorem~\ref{thm:szb2} below which satisfies all these properties, and additionally coincides with standard Scott-Zhang operator $J_\ell^2\colon H^1(\Omega)\to \SS^2(\widehat\TT_\ell)$ on the set $\widetilde B_{\ell+1}^1$ from Definition~\ref{def:b1tilde}. This allows us to use the standard operators $J_\ell^2$ for the construction of the sets $B_\ell^1$ below and to use the modified operators $S_\ell^2$ in the proof of the theorem.

\begin{theorem}\label{thm:stableB1}
With the spaces from Definition~\ref{def:b1tilde} and the Scott-Zhang operator $J_\ell^2\colon H^1(\Omega)\to \SS^2(\widehat\TT_\ell)$ from Lemma~\ref{lem:sz2}, define $B_0^1:=\widetilde B_0^1$ and
for $\ell\geq 1$
\begin{align*}
B_\ell^1:=\set{(1-J_{\ell-1}^2)v_0}{v_0\in \widetilde B_\ell^1}.
\end{align*}
Define $B^1:=\bigcup_{\ell\in\N}B_\ell^1$.  Then, $B^1$ is a Riesz bases of $\overline{\bigcup_{k\in\N}\SS^{2}_{0}(\TT_k)}\subseteq H^1_0(\Omega)$, i.e., 
\begin{align}\label{eq:rieszB1}
\begin{split}
\norm{\sum_{v\in B^1} \alpha_v v}{H^{1}(\Omega)}&\simeq \Big(\sum_{v\in B^1} \alpha_v^2\Big)^{1/2}.
\end{split}
\end{align}
Moreover, ${\rm diam}({\rm supp}(v))\simeq C_{\rm mesh}^{-\ell} $ for all $v\in B_\ell^1$ and there holds  
\begin{align}\label{eq:scaling1}
\norm{v}{H^{s}({\rm supp}(v))}&\simeq C_{\rm mesh}^{-s\ell}\norm{v}{L^2({\rm supp}(v))}\quad\text{\rm for all } v\in B_\ell^1\text{ and all } 0\leq s<3/2.
\end{align}
\end{theorem}
\begin{proof}
Due to the range of $S_\ell^2$, standard inverse estimates prove
\begin{align*}
\norm{S^2_\ell u}{H^{5/4}(\Omega)}&\lesssim C_{\rm mesh}^{\ell/4}\norm{S^2_\ell u}{H^1(\Omega)},\\
\norm{S^2_\ell u}{H^{1}(\Omega)}&\lesssim C_{\rm mesh}^{\ell/4}\norm{S^2_\ell u}{H^{3/4}(\Omega)}
\end{align*}
and obviously the ranges $\SS^{2}(\widehat \TT_\ell)$  form a dense and nested sequence of subspaces of $H^1(\Omega)$.
Lemma~\ref{lem:gelfand} confirms that $H^{3/4}(\Omega)$ is the dual space of $H^{5/4}(\Omega)$ with respect to the $H^1(\Omega)$-scalar product. 
Hence, we are in the position to apply~\cite[Theorems~3.1\&3.2]{Dahmen} to prove
\begin{align}\label{eq:Did}
\norm{u}{H^1(\Omega)}^2\simeq \sum_{\ell=0}^\infty \norm{(S^2_\ell-S^2_{\ell-1})u}{H^1(\Omega)}^2,
\end{align}
where we define $S^2_{-1}:=0$.
The identity~\eqref{eq:S2id} implies that for $v_0\in \widetilde B_k^1$	 
there holds (with $J_k^2$ from Theorem~\ref{thm:szb2} below)
\begin{align*}
 (1-J_{k-1}^2)v_0= (1-S^2_{k-1})v_0
\end{align*}
and Theorem~\ref{thm:szb2} shows
\begin{align*}
(S^2_\ell-S^2_{\ell-1})(1-S^2_{k-1})v_0&=S^2_\ell v_0 - S^2_\ell S^2_{k-1}v_0 - S^2_{\ell-1}v_0+ S^2_{\ell-1}S^2_{k-1}v_0 \\
&=\begin{cases}
S^2_\ell v_0- S^2_{\ell-1}v_0 =0 &k<\ell,\\ 
(1-S^2_{\ell-1}) v_0& k=l,\\ 
0 & k>\ell.
\end{cases}
\end{align*}
Thus, writing $w=\sum_{\ell\in\N_0}\sum_{v\in B^1_\ell} \alpha_v v$, we get with~\eqref{eq:Did}
\begin{align}\label{eq:firstriesz}
\norm{w}{H^{1}(\Omega)}^2&\simeq \sum_{\ell\in\N_0}\norm{(S^2_\ell-S^2_{\ell-1})w}{H^1(\Omega)}^2\simeq \sum_{\ell\in\N_0}\norm{\sum_{v\in B^1_\ell} \alpha_v v}{H^1(\Omega)}^2.
\end{align}
For $v\in B^1_\ell$, there exists a unique $v_0\in \widetilde B^1_\ell$ such that $v=(1-J_{\ell-1}^2)v_0$. 
For $T\in\widehat\TT_{\ell-1}$, scaling arguments and stability of $J_{\ell-1}$ show $\norm{\cdot}{H^1(T)}\lesssim \norm{(1-J_{\ell-1}^2)(\cdot)}{H^1(T)}\lesssim
\norm{\cdot}{H^1(\bigcup\omega(T,\widehat\TT_{\ell-1}))}$ on
the finite dimensional space ${\rm span}\set{v|_{\bigcup\omega(T,\widehat\TT_{\ell-1})}}{v\in \widetilde B^1_\ell}$
(with uniform constants). This implies
\begin{align*}
\norm{w}{H^1(\Omega)}^2&\simeq  \sum_{\ell\in\N}\norm{(1-J_{\ell-1}^2)\sum_{v\in B^1_\ell} \alpha_v v_0}{H^1(\Omega)}^2
\simeq  \sum_{\ell\in\N}\norm{\sum_{v\in B^1_\ell} \alpha_v v_0}{H^1(\Omega)}^2.
\end{align*}
Since for all $T\in\widehat\TT_{\ell-1}$, the set $\set{v_0|_T}{v_0\in \widetilde B^1_\ell}$ is linear independent, a scaling
argument shows
\begin{align}\label{eq:locbasis1}
\begin{split}
 \norm{\sum_{v\in B^1_\ell} \alpha_v v_0}{H^1(T)}^2&
 \simeq \sum_{\stacktwo{v\in B^1_\ell}{  v_0|_T\neq 0}} \alpha_v^2.
 \end{split}
\end{align}
Another norm equivalence argument shows $1=\norm{v_0}{H^1(\Omega)}\simeq \norm{(1-J_{\ell-1}^2)v_0}{H^1(\Omega)}$ for all
$v_0\in\widetilde B_\ell^1$ and hence
\begin{align*}
\norm{w}{H^1(\Omega)}^2\simeq \sum_{\ell\in\N}\sum_{v\in B^1_\ell} \alpha_v^2\norm{ v}{H^1({\rm supp}(v))}^2.
\end{align*}
Since $B_0^1$ contains only finitely many linear independent functions, we also obtain

\begin{align*}
\norm{w}{H^{1}(\Omega)}^2\simeq \sum_{\ell\in\N_0}\sum_{v\in B^1_\ell} \alpha_v^2\norm{ v}{H^1({\rm supp}(v))}^2\simeq \sum_{v\in B^1} \alpha_v^2.
\end{align*}
Altogether, the operator $\iota\colon \ell_2(B^1)\to H^{1}(\Omega)$, $\iota(\alpha):=\sum_{v\in B^1}\alpha_v v$ is
bounded and has a bounded inverse on its closed range. Since $J_\ell^2$ retains homogeneous boundary values (Lemma~\ref{lem:sz2}), 
the range of $\iota$ is dense in $\overline{\bigcup_{k\in\N}\SS^{2}_0(\TT_k)}\subseteq H^1(\Omega)$ and hence $\iota$ is bijective.
This shows that $B^1$ is a Riesz basis of  $\overline{\bigcup_{k\in\N}\SS^{2}_0(\TT_k)}\subseteq H^1_0(\Omega)$.

The scaling estimate~\eqref{eq:scaling1} can be proved as follows: Let $v=(1-J_{\ell-1}^2)v_0\in B_\ell^1$ for some $v_0\in \widetilde B_\ell^1$ and let $\omega:={\rm supp}(v)$. 
The approximation property and the projection property of $J_{\ell-1}^2$ show 
\begin{align*}
\norm{v}{L^2(\omega)}\lesssim C_{\rm mesh}^{-s\ell}\norm{v}{H^s(\omega)}
\end{align*}
for all $0\leq s<3/2$.
The converse estimates $\norm{v}{ L^2(\omega)}\gtrsim C_{\rm mesh}^{-s\ell}\norm{v}{H^s(\omega)}$
for $0\leq s<3/2$ follow from  standard inverse estimates. 
This concludes~\eqref{eq:scaling1} and thus concludes the proof.
\end{proof}

The following definition and theorem construct a Riesz basis for the pressure variable of the Stokes problem in $L^2_\star(\Omega)$.
\begin{definition}\label{def:b0tilde}
 We define the basis $\widetilde B_0^0$ as an arbitrary basis of $\SS^1_\star(\widehat\TT_0)$. For $\ell\geq 1$, define
 \begin{align*}
  \widetilde B_\ell^0:=\set{\frac{v}{\norm{v}{L^2(\Omega)}}}{ v\in
  \Big(\bigcup_{k\in\N}\SS^{1}_{B}(\TT_k\setminus\TT_{k-1})\Big)\cap (\SS^{1}(\widehat\TT_\ell)\setminus
\SS^{1}(\widehat\TT_{\ell-1}))},
 \end{align*}
\end{definition}

\begin{theorem}\label{thm:stableB0}
With the spaces from Definition~\ref{def:s2} and the modified Scott-Zhang operator from Definition~\ref{def:sz1} $J_{\ell-1}^1$, define $B_0^0=\widetilde B_0^0$
and for $\ell\geq 1$
\begin{align*}
B_\ell^0:=\set{(1-J_{\ell-1}^1)v_0}{v_0\in\widetilde B_\ell^0 }.
\end{align*}
Define $B^0:=\bigcup_{\ell\in\N}B_\ell^0$.  Then, $B^0$ is Riesz bases of $\overline{\bigcup_{k\in\N}\SS^{1}_\star(\TT_k)}\subseteq L^2_\star(\Omega)$, i.e., 
\begin{align}\label{eq:rieszB0}
\begin{split}
\norm{\sum_{v\in B^0} \alpha_v v}{L^2(\Omega)}&\simeq \Big(\sum_{v\in B^0} \alpha_v^2\Big)^{1/2}.
\end{split}
\end{align}
Moreover, ${\rm diam}({\rm supp}(v))\simeq C_{\rm mesh}^{-\ell} $ for all $v\in B_\ell^0$ and there holds for all $v\in B_\ell^0$ 
\begin{align}\label{eq:scaling0}
\begin{split}
\norm{v}{H^{s}({\rm supp}(v))}&\simeq C_{\rm mesh}^{-s\ell}\norm{v}{L^2({\rm supp}(v))}\quad\text{for all }0\leq s<3/2,\\
\norm{v}{\widetilde H^{-s}({\rm supp}(v))}&\simeq C_{\rm mesh}^{s\ell}\norm{v}{L^2({\rm supp}(v))}\quad\text{for all }0\leq s\leq 1.
\end{split}
\end{align}
\end{theorem}
\begin{proof}
The idea of the proof is identical to that of Theorem~\ref{thm:stableB1}; we would like to use a Scott-Zhang projector but require some additional properties provided by the projection from Theorem~\ref{thm:szb1}.
	First, we note that stability of the $J^1_{\ell-1}$ and scaling argument prove the estimate $\norm{v}{L^2(\Omega)}\simeq 1$ uniformly
	for all $v\in B^0$. Moreover, Lemma~\ref{lem:sz1} shows immediately that $(1-J_{\ell-1}^1)v_0\in L^2_\star(\Omega)$.
Again, we use~\cite{Dahmen} with the operators $(S_\ell^1)_{\ell\in\N_0}$ from Theorem~\ref{thm:szb1}:
The operators $S^1_\ell$ are uniformly $L^2(\Omega)$-bounded and satisfy for all $\ell\leq k$.
\begin{align}\label{eq:iter}
 S_\ell^1S_k^1&=S_\ell^1.
\end{align}
Moreover, their ranges $\SS^{1}(\widehat \TT_\ell)$ form a dense and nested sequence of subspaces of $L^2(\Omega)$.
Theorem~\ref{thm:szb1} confirms the approximation estimates
\begin{align*}
\norm{(1-S^1_\ell) u}{\widetilde H^{-1/4}(\Omega)}&\lesssim C_{\rm mesh}^{-\ell/4}\norm{u}{L^2(\Omega)},\\
\norm{(1-S^1_\ell) u}{L^2(\Omega)}&\lesssim C_{\rm mesh}^{-\ell/4}\norm{u}{H^{1/4}(\Omega)}
\end{align*}
as well as uniform boundedness $S^1_\ell\colon \widetilde H^{-1/4}(\Omega)\to \widetilde H^{-1/4}(\Omega)$.
Standard inverse estimates prove
\begin{align*}
\norm{S^1_\ell u}{H^{1/4}(\Omega)}&\lesssim C_{\rm mesh}^{\ell/4}\norm{S^1_\ell u}{L^2(\Omega)},\\
\norm{S^1_\ell u}{L^2(\Omega)}&\lesssim C_{\rm mesh}^{\ell/4}\norm{S^1_\ell u}{\widetilde H^{-1/4}(\Omega)}.
\end{align*}
Therefore, we may apply~\cite[Theorems~3.1\&3.2]{Dahmen} to prove
\begin{align}\label{eq:Did0}
\norm{u}{L^2(\Omega)}^2\simeq \sum_{\ell=0}^\infty \norm{(S^1_\ell-S^1_{\ell-1})u}{H^1(\Omega)}^2,
\end{align}
where we define $S^1_{-1}:=0$.
The identity~\eqref{eq:Sid} implies that for $v_0\in \widetilde B_\ell^0$ for $\ell\geq 1$, there holds
\begin{align*}
 J_{k}^1v_0= S^1_{k}v_0\quad\text{for all }k\geq \ell-1.
\end{align*}
The identity~\eqref{eq:iter} shows for $v_0\in \widetilde B_k^0$
\begin{align*}
(S^1_\ell-S^1_{\ell-1})(1-S^1_{k-1})v_0
&=S^1_\ell v_0 - S^1_\ell S^1_{k-1}v_0 - S^1_{\ell-1}v_0+ S^1_{\ell-1}S^1_{k-1}v_0 \\
&=\begin{cases}
S^1_\ell v_0- S^1_{\ell-1}v_0 =0 &k<\ell,\\ 
(1-S^1_{\ell-1}) v_0=(1-J_{\ell-1}^1)v_0& k=\ell,\\ 
0 & k>\ell.
\end{cases}
\end{align*}
As in the proof of Theorem~\ref{thm:stableB1}, we obtain
\begin{align*}
	\norm{w}{L^2(\Omega)}^2\simeq \sum_{\ell\in\N_0}\sum_{v\in B^0_\ell} \alpha_v^2\norm{ v}{L^2({\rm supp}(v))}^2
\end{align*}
for all $w=\sum_{\ell\in\N_0}\sum_{v\in B^0_\ell} \alpha_v v$.
Therefore, the operator $\iota\colon \ell_2(B^0)\to L^2_\star(\Omega)$, $\iota(\alpha):=\sum_{v\in B^0}\alpha_v v$ is
bounded and has a bounded inverse on its closed range. Obviously, the range is dense in $\overline{\bigcup_{k\in\N}\SS^{1}_\star(\TT_k)}\subseteq L^2(\Omega)$ and
hence $\iota\colon \ell_2(B^0)\to \overline{\bigcup_{k\in\N}\SS^{1}_\star(\TT_k)}$ is bijective.
This concludes that $B^0$ is a Riesz basis of  $\overline{\bigcup_{k\in\N}\SS^{1}_\star(\TT_k)}\subseteq L^2_\star(\Omega)$.

The scaling estimates~\eqref{eq:scaling0} can be proved as follows: Let $v=(1-J_{\ell-2}^1)v_0\in B_\ell^0$ for some $v_0\in \widetilde B_\ell^0$ and let $\omega:={\rm supp}(v)$. 
By construction of $J_{\ell-1}^1$, we see that $v=(1-J_{\ell-1}^1)v_0$ has integral mean zero on its support. 
Therefore, a Poincar\'e estimate proves
\begin{align*}
\norm{v}{\widetilde H^{-s}(\omega)} \lesssim C_{\rm mesh}^{-s\ell}\norm{v}{L^2(\omega)} 
\end{align*}
for all $0\leq s\leq 1$.
The approximation property~\eqref{eq:sza} and the projection property of $J_{\ell-1}^1$ show 
\begin{align*}
\norm{v}{L^2(\omega)}\lesssim C_{\rm mesh}^{-s\ell}\norm{v}{H^s(\omega)}
\end{align*}
for all $0\leq s<3/2$.
The converse estimates $\norm{w}{ L^2(\omega)}\gtrsim C_{\rm mesh}^{-s\ell}\norm{w}{H^s(\omega)}$ for $0\leq s<3/2$ as well as $\norm{v}{\widetilde H^{-s}(\omega)} \gtrsim C_{\rm mesh}^{-s\ell}\norm{v}{L^2(\omega)} $
for $0\leq s\leq 1$ follow from  standard inverse estimates. 
This concludes the proof.
\end{proof}

\begin{lemma}\label{lem:graded}
 Let $\TT\in\T_{\rm grad}$ and $\ell\in\N$. Let $ T\in \TT\setminus \widehat \TT_\ell$ such that $\TT|_{T}$ is a strictly finer local refinement of $\widehat \TT_\ell$.
 Then, $\TT|_{\omega^{D_{\rm grad}}(T,\widehat\TT_\ell)}$ is a local refinement of $\widehat\TT_\ell$.
\end{lemma}
\begin{proof}
Let $L\in\N$ denote the level of the elements in $\widehat\TT_\ell$. By assumption, there holds ${\rm level}(T)>L$ and hence the assertion is a direct consequence of~\eqref{eq:graded}.
\end{proof}

\begin{lemma}\label{lem:span}
 Given a mesh $\TT$ with~\eqref{eq:graded} for $D_{\rm grad}\geq 3$, there holds
 \begin{align*}
  \SS^2_0(\TT)={\rm span}(B^1\cap \SS^2(\TT))\quad\text{and}\quad  \SS^1_\star(\TT)={\rm span}(B^0\cap \SS^1_\star(\TT)).
 \end{align*}
\end{lemma}
\begin{proof}
We note that $\SS^{2}_0(\TT)={\rm span}\set{v_0\in\bigcup_{\ell=0}^\infty\widetilde B_\ell^1}{v_0\in\SS^{2}_0(\TT)}$.
For each $v_0 \in \SS^{2}_0(\TT)\cap \widetilde B_\ell^1$, we note that $J_{\ell-1}v_0$ is supported on $\omega({\rm supp}(v_0),\widehat\TT_{\ell-1})$.
Since $v_0\in\SS^2_0(\TT)\cap\widetilde B_\ell^1$ implies that $\TT|_{{\rm supp}(v_0)}$ is a true local refinement of
$\widehat\TT_{\ell-1}$, there exists at least one  $T\in\TT\setminus\widehat\TT_{\ell-1}$ with $T\subseteq {\rm supp}(v_0)$.
By definition of $v=(1-J_{\ell-1}^2)v_0$, we know that ${\rm supp}(v)\subseteq \bigcup \omega^2(T,\widehat\TT_{\ell-1})$. Lemma~\ref{lem:graded} proves that
$\TT|_{\omega^2(T,\widehat\TT_{\ell-1})}$ is a local refinement of $\widehat\TT_{\ell-1}$ and hence $v\in \SS^2_0(\TT)$.
This concludes $\SS^{2}_0(\TT)={\rm span}\set{v\in B^1}{v\in\SS^{2}_0(\TT)}$.

The same argument works for $\SS^1_\star(\TT)$ and concludes the proof.
\end{proof}

\section{Proof of (A1)--(A4)}\label{sec:proofA1A4}
The abstract theory developed in the previous sections allows us to prove optimality of the adaptive algorithm for the stationary
Stokes problem.

\subsection{Proof of (A1), (A2), and (A4)}
The proofs of~(A1),~(A2), and~(A4) will not be surprising to experts as they are well-known in the literature and can be found in~\cite{gantstokes}.
\begin{proof}[Proof of (A1)\&(A2)]
 The proof follows from standard arguments as for the Poisson problem (see, e.g.,~\cite{ckns}) and can be found in condensed form in~\cite[Lemma~3.2]{gantstokes}.
\end{proof}

\begin{proof}[Proof of (A4)]
 The property~(A4) can be found in~\cite[Lemma~3.1]{gantstokes}.
\end{proof}

\subsection{Proof of (A3)}
The main innovation of this paper is the proof of general quasi-orthogonality~(A3). The strategy of proof for this section is the following:
\begin{itemize}
 \item Approximate the infinite stiffness matrix of $a(\cdot,\cdot)$ with respect to the basis from Section~\ref{section:basis} by a banded matrix.
 \item Use the tools from Section~\ref{section:lu} to show that the banded matrix has an inversely bounded $LU$-factorization.
 \item Use Theorem~\ref{thm:luqo2} to show that general quasi-orthogonality~(A3) holds.
\end{itemize}
 To that end, we first need to show that we can restrict the problem to the space $\widetilde \XX:=\overline{\bigcup_{\ell\in\N_0}\XX_\ell}\subseteq \XX$.
 To that end, the result~\cite[Lemma~4.2]{msv} shows that $a(\cdot,\cdot)$ satisfies the inf-sup condition~\eqref{eq:infsup} even on $\widetilde\XX$ and that $(u,p)\in\widetilde\XX$.
To fit the problem into our abstract framework, we choose the following Riesz basis of $\widetilde\XX$ from Theorems~\ref{thm:stableB1}\&\ref{thm:stableB0}: 
\begin{align*}
	\bB:=\set{(v,0,0)}{v\in B^1}\cup \set{(0,v,0)}{v\in B^1}\cup \set{(0,0,w)}{w\in B^0}.
\end{align*}
We recall that $\XX_\ell\subseteq \XX_{\ell+1}\subset \XX$ are nested finite dimensional spaces generated 
by the adaptive algorithm described in Section~\ref{section:mesh}. We introduce the level function $L(\bw)=\ell$ for all $\bw\in (B_\ell^1)^2\times B_\ell^0$ and also $L(v)=\ell$ for all $v\in B_\ell^1\cup B_\ell^0$.
We order the functions in $\bB$ such that $\XX_\ell={\rm span}\{\bw_1,\bw_2,\ldots,\bw_{N_\ell}\}$ for particular $N_\ell\in\N$ and all $\ell\in\N$ (note that this is possible due to Lemma~\ref{lem:span}).

The proofs below will use several distance functions to capture the sparsity structure of the involved matrices. Those distance functions are
defined in~\cite{fembemopt} and we recall them in the following:
\begin{definition}\label{def:metric}
 For $B:=B^0\cup B^1$, define the following functions:
 \begin{itemize}
 \item  $\delta\colon B\times B\to \{0,1\}$ is defined by $\delta(v,w)=1$ if $v\neq w$ and $\delta(v,w)=0$ if $v=w$.
 \item  $\delta_k\colon B\times B\to \N$ is defined by 
   \begin{align*}
   \delta_k(v,w):=\min\set{n\in\N}{&\exists T_1,\ldots,T_{n}\in\widehat\TT_{k},\,{\rm mid}(v)\in T_1,
   \,{\rm mid}(w)\in T_n,\\ &T_i\cap T_{i+1}\neq \emptyset,\,i=1,\ldots,n-1},
  \end{align*}
  where ${\rm mid}(\cdot)$ denotes the barycenter of the support of the function. Note: This metric measures the distance with respect to
  the element size on $\widehat\TT_k$.
  \item $d_1\colon B\times B\to \N$ is defined by
  \begin{align*}
   d_1(v,w):=\delta_{\min\{L(v),L(w)\}}(v,w).
  \end{align*}
\item Given $\beta>0$, $d_2\colon B\times B\to [0,\infty)$ is defined by
\begin{align*}
 d_2(v,w):= \delta(v,w)+\beta|L(v)-L(w)| + \log(\delta(v,w)+d_1(v,w)).
\end{align*}
Note: This metric combines the physical distance of the supports with the level distance. The logarithm is necessary to ensure the triangle inequality.
\item Given $\gamma>0$, $d_3\colon B\times B\to [0,\infty)$ is defined by
\begin{align*}
 d_3(v,w):= \begin{cases}
             \gamma^{\max\{L(v),L(w)\}} & L(v)\neq L(w),\\
             \delta(v,w)+d_1(v,w)-1& L(v) = L(w).
            \end{cases}
\end{align*}
Note: This metric is designed for block-diagonal matrices, where $d_3(v,w)=\infty$ for $L(v)\neq L(w)$.
The quantity $\gamma^{\max\{L(v),L(w)\}}$ is just a sufficiently large placeholder to ensure the triangle inequality and avoid the use of $\infty$.
 \end{itemize}
It is shown in~\cite[Section~4]{fembemopt} that for sufficiently large $\beta,\gamma>0$, $d_2$ and $d_3$ are metrics on $\bB$. If $\bB$ is identified with $\N$, $d_3$ even satisfies~\eqref{eq:Mdist}.
\end{definition}

Moreover, we will consider two different block-structures (Definition~\ref{def:blockstruct}). The first one results from the adaptive
spaces $\XX_\ell$, $\ell\in\N$ which induce the adaptive block structure $N_1,N_2,\ldots$ as described above.
The second block structure is induced by reordering $\bB$ such that $L(\bw_{\pi(i)})\leq L(\bw_{\pi(j)})$ for all $i\leq j$ and some permutation $\pi$.
This permutation defines a unique permutation matrix $P\in \{0,1\}^{\N\times\N}$ defined by $P_{ij}=1$ if and only if $i=\pi(j)$.
This yields the uniform block-structure $n_1,n_2,\ldots$ such that $\set{\bw_{k_i}}{i=n_r,\ldots,n_{r+1}-1}=\set{\bw\in\bB}{L(\bw)=r}$.

\medskip

The following lemma states that under~\eqref{eq:graded}, an adaptive space $\XX_\ell$ which contains a basis function $\bw\in\bB$
also contains lower level basis functions which are \emph{close} to $\bw$. For brevity of presentation of the following lemma, we identify $\bw\in\bB$ with its
unique non-zero component, e.g, $\bw=(0,w,0)$.
\begin{lemma}\label{lem:gradednested}
 Let  $\v,\w\in\B$ such that $\v\in \XX_\ell$ for some $\ell\in\N_0$.
Let $D_{\rm grad}>0$ and assume~\eqref{eq:graded} for all $(\TT_\ell)_{\ell\in\N_0}$.
Then, $L(\w)<L(\v)$ and $\delta_{L(\w)}(\v,\w)\leq D_{\rm grad}-C_{\rm grad}$ imply $\w\in\XX_\ell$.
The constant $C_{\rm grad}\geq 1$ depends only on $\TT_0$ and $C_{\rm mesh}$. 
\end{lemma}
\begin{proof}
Define 
\begin{align*}
 C:=\sup_{k\in\N}\max_{w\in\bigcup_{j=k}^\infty (B_j^1\cup B_j^0)}\#\set{T\in\widehat\TT_k}{T\subseteq {\rm supp}(w)}.
\end{align*}
By definition of $\B$, we see that $C<\infty$ in terms of $\TT_0$ and $C_{\rm mesh}$. Let ${\rm level}(\widehat\TT_k)$ denote the uniform level of elements in $\widehat\TT_k$, $k\in\N_0$. 
We prove the statement by contradiction, i.e., we show that $\bw\notin \XX_\ell$ and $L(\w)<L(\v)$ lead to a contradiction.
To that end, find $T,T^\prime\in \TT_\ell$ with $T\cap {\rm supp}(\w)\neq \emptyset$ and $T^\prime\cap {\rm supp}(\v)\neq \emptyset$.
Since $\bw\notin \XX_\ell$,
we may choose $T$ such that ${\rm level}(T)< {\rm level}(\widehat\TT_{L(\bw)})$.
Assumption~\eqref{eq:graded} implies that all $T^{\prime\prime}\in \omega^{D_{\rm grad}}(T,\TT_\ell)$
satisfy ${\rm level}(T^{\prime\prime})\leq {\rm level}(T)+1\leq{\rm level}(\widehat\TT_{L(\bw)}) $. This means that $\widehat\TT_{L(\bw)}|_{\omega^{D_{\rm grad}}(T,\TT_\ell)}$ is a local refinement of $\TT_\ell$.
With $\chi_T$ denoting the function which is one on $T$ and zero elsewhere, we get
\begin{align*}
 \delta_{L(\bw)}(\chi_T,\chi_{T^\prime})\leq 2C + \delta_{L(\bw)}(\v,\w).
\end{align*}
Thus, $\delta_k(\v,\w)\leq {D_{\rm grad}}-2C$ implies $\delta_{L(\bw)}(\chi_{T},\chi_{T^\prime})\leq D_{\rm grad}$.
Since we argued above that $\widehat\TT_{L(\bw)}|_{\omega^{D_{\rm grad}}(T,\TT_\ell)}$ is a local refinement
of $\TT_\ell$, $\delta_{L(\bw)}(\chi_{T},\chi_{T^\prime})\leq D_{\rm grad}$ implies $T^\prime \in \omega^{{D_{\rm grad}}}(T,\TT_\ell)$.
This shows ${\rm level}(T^\prime)\leq {\rm level}(T)+1\leq  {\rm level}(\widehat\TT_{L(\bw)})$ which contradicts $\bv\in\XX_\ell$ and $L(\bv)> L(\bw)$.
This concludes the proof with $C_{\rm grad}:=2C$.
\end{proof}
The next lemma shows that the adaptive and the uniform block-structure are compatible with each other
in the sense that being block-triangular is (sometimes) invariant under permutation between the two structures.
\begin{lemma}\label{lem:blockstruct}
 Let $L\in\R^{\N\times\N}$ with $L^{-1}\in \BB(d_2,b)$ for some bandwidth $b\in\N$ and let $L$ be lower-triangular with identity diagonal blocks 
 $L|_{\{n_r,\ldots n_{r+1}-1\}\times \{n_r,\ldots n_{r+1}-1\}}= I$ for all $r\in\N$.
If all $(\TT_\ell)_{\ell\in\N}$ satisfy~\eqref{eq:graded} for some sufficiently large $D_{\rm grad}$ depending only on $b$, then
 $P L P^T$ is block-lower-triangular with respect to the adaptive block structure $N_1,N_2,\ldots$.
\end{lemma}
\begin{proof}
 We show the equivalent statement that $(P L P^T)^{-1} = P L^{-1} P^T$ is block-lower-triangular.
 Under the assumptions on $L^{-1}$, the fact $(L^{-1})_{\pi^{-1}(i)\pi^{-1}(j)}\neq 0$ implies either $\pi^{-1}(i)=\pi^{-1}(j)$ and hence $i=j$
or $\pi^{-1}(i)\geq n_r> \pi^{-1}(j)$ for some $r\in\N$ and hence $L(\bw_{\pi^{-1}(j)})<L(\bw_{\pi^{-1}(i)})$. In the latter case,  $L^{-1}\in \BB(d_2,b)$ shows additionally that
$d_2(\bw_{\pi^{-1}(i)},\bw_{\pi^{-1}(j)})\leq b$ and hence
$d_{L(\bw_{\pi^{-1}(j)})}(\bw_{\pi^{-1}(i)},\bw_{\pi^{-1}(j)})\lesssim \exp(b)$. Thus, for sufficiently large $D_{\rm grad}$, Lemma~\ref{lem:gradednested} applies and 
shows $\bw_{\pi^{-1}(j)}\in \XX_\ell$ for any $\ell\in\N$ with $\bw_{\pi^{-1}(i)}\in\XX_\ell$. In other words,
 we showed that $(PL^{-1}P^T)_{ij}\neq 0$ implies $j\leq N_\ell$ for all $\ell\in\N$ with $i\leq N_\ell$ and hence
 $PL^{-1} P^T$ is block-lower-triangular with respect to the adaptive block structure $N_1,N_2,\ldots$. This concludes the proof.
\end{proof}

\begin{theorem}\label{thm:qosum}
 Given $\eps>0$ and under all previous assumptions, there exists $D_{\rm grad}>0$ sufficiently large such that 
 the solutions~\eqref{eq:solutions} of the stationary Stokes problem~\eqref{eq:stokes} satisfy general quasi-orthogonality~(A3).
\end{theorem}

\begin{proof}
\emph{Step~1:}
With the basis $\bB$, define the matrix $A\in\R^{\N\times \N}$ by
\begin{align*}
 A_{ij}:=a(\bw_j,\bw_i). 
\end{align*}
With the two different block structures defined above, we consider the permuted matrix
\begin{align*}
 \widetilde A_{ij}:= A_{\pi(i)\pi(j)}\quad\text{or}\quad  \widetilde A=P^TAP.
\end{align*}
Since $\bB$ is a Riesz basis,~\eqref{eq:infsup}--\eqref{eq:discinfsup} of $a(\cdot,\cdot)$ imply that $A$ satisfies~\eqref{eq:matinfsup}.
By definition of $\bB$, we observe that $\set{\bw_{\pi(i)}}{i=1,\ldots,n_{k+1}-1}$ spans the space
\begin{align*}
\Big(\SS^2_0(\widehat\TT_k)^2\times \SS^1_\star(\widehat\TT_k)\Big)\cap \bigcup_{\ell\in\N}\Big(\SS^2_0(\TT_\ell)^2\times \SS^1_\star(\TT_\ell)\Big)=\SS^2_0(\TT)^2\times \SS^1_\star(\TT),
\end{align*}
where $\TT\in\T$ is the unique mesh which contains exactly the nodes and edges that appear in $\widehat\TT_k$ and additionally in at least one (and thus all the subsequent)
of the adaptive meshes $(\TT_\ell)_{\ell\in\N_0}$. Therefore,~\cite{boffi} shows that~\eqref{eq:discinfsup} is
satisfied for $\XX_\TT:=\SS_0^2(\TT)^2\times \SS^1_\star(\TT)$ (for all $r\in\N$) and hence~\eqref{eq:matinfsup} holds also for the uniform block-structure on $\widetilde A$.

\emph{Step~2:}
Lemmas~\ref{lem:Ascaling}--\ref{lem:Bscaling} below show that there exists $ \widetilde A^\eps$ which is block-banded (with the uniform block-structure $n_1,n_2,\ldots$) for some bandwidth $b$ such that
\begin{align*}
 \norm{ \widetilde A- \widetilde A^\eps}{2}\leq \eps.
\end{align*}
Moreover, if we identify $i\mapsto \bw_i$, there holds $ \widetilde A^\eps\in\BB(d_2)$ with the metric $d_2(\cdot,\cdot)$ from Definition~\ref{def:metric}.
Choosing $\eps>0$ sufficiently small, we ensure that $\widetilde A^\eps$ satisfies~\eqref{eq:matinfsup} as well.
Thus, Theorem~\ref{thm:blockLU} shows that there exists an approximate block-$LDU$-factorization with $\norm{ \widetilde A-LDU}{2}\leq 2\eps$, i.e., 
\begin{align}\label{eq:approxMat}
 \norm{A -  (PLP^T)(PDP^T)(PU P^T)}{2}\leq 2\eps.
\end{align}
The factors $L^{-1},D,U^{-1}\in\BB(d_2)$ are block-banded (with respect to the uniform block structure $n_1,n_2,\ldots$), bounded, and have identity diagonal blocks.
Thus, Lemma~\ref{lem:blockstruct} applies and shows that $PLP^T$ and $PUP^T$ are block-triangular with respect to the adaptive block-structure $N_1,N_2,\ldots$.

\emph{Step~3:}
Since $D$ is block-diagonal and satisfies~\eqref{eq:matinfsup} (according to Theorem~\ref{thm:blockLU}, with the uniform block structure $n_1,n_2,\ldots$), we also have $PDP^T\in\BB(d_3)$ for the metric $d_3(\cdot,\cdot)$ from Definition~\ref{def:metric} (again identifying $i\mapsto \bw_i$). This shows $PDP^T\in \JJ(d_3)$. Since $A$ satisfies~\eqref{eq:matinfsup}, the estimate~\eqref{eq:approxMat} together with the fact that $PLP^T$ and $PUP^T$ are block-triangular and inversely bounded (uniformly in $\eps$, see Theorem~\ref{thm:blockLU}) shows that $PDP^T$ satisfies~\eqref{eq:matinfsup} for the adaptive block-structure $N_1,N_2,\ldots$  as long as 
$\eps>0$ is sufficiently small.
Thus, Theorem~\ref{thm:LU} implies the existence of a bounded block-$LU$-factorization $PDP^T = \widetilde L\widetilde U$ (with respect to the adaptive block-structure) with inversely bounded factors $\widetilde L,\widetilde U\in\R^{\N\times\N}$. Altogether, we found a matrix $A^\eps:= P \widetilde A^\eps P^T$ with  block-$LU$-factorization $A^\eps = (PLP^T\widetilde L ) (\widetilde U P U P^T)$ with inversely bounded factors. Finally, Theorem~\ref{thm:luqo2} with $\widetilde \XX$ instead of $\XX$ applies and concludes the proof.

\end{proof}

\section{Technical results and auxiliary lemmas}\label{sec:tech}
This final section covers some of the more technical aspects of the work.
\subsection{Almost bandedness of differential operator matrices}\label{section:banded}
In this section, we prove that the operator matrices coming from the Stokes problem in the Riesz basis are arbitrarily 
close to banded matrices in the sense of Definition~\ref{def:banded}. The results will not surprise anyone familiar with wavelet analysis but 
cannot be found in the literature directly as we use a very particular Riesz basis. The main proof technique is to exploit locality of the basis functions as well as 
strengthened Cauchy-Schwarz inequalities between the levels.
\begin{definition}\label{def:sobslob}
 Define the Sobolev-Slobodeckij semi norm 
 \begin{align*}
 |v|_{H^s(\omega)}^2:= \int_\omega\int_\omega |v(x)-v(y)|/|x-y|^{2+2s}\,dx\,dy\quad\text{for } 0<s<1.
 \end{align*}
 For $s=\nu+r\in\R$ with $\nu\in\N$ and $s\in (0,1)$, define $|\cdot|_{H^s(\omega)}:=|\nabla^\nu(\cdot)|_{H^r(\Omega)}$, where $\nabla^\nu$ 
 denotes the tensor of all partial derivatives of order  $\nu$. 
 As shown in~\cite{heu}, $\norm{\cdot}{H^\nu(\omega)}+|\cdot|_{H^r(\omega)}$  is equivalent to the $H^s$-norm obtained  via (real) interpolation.
 The norm equivalence constants depend only on the shape of $\omega$.
\end{definition}
\begin{lemma}\label{lem:restriction}
Let $v\in H^s(\omega)$ for some $0\leq s<1/2$ and some Lipschitz domain $\omega\subseteq \Omega$.
Then, there holds for the extension of $v$ to $\Omega$ by zero
\begin{align*}
 \norm{v}{H^s(\Omega)}\leq C_{\rm res} \norm{v}{H^s(\omega)}.
\end{align*}
The constant $C_{\rm res}>0$ does only depend on the shape of $\omega$.
\end{lemma}
\begin{proof}
There holds
 \begin{align*}
  |v|_{H^s(\Omega)}^2-|v|_{H^s(\omega)}^2 &= 2\int_{\Omega\setminus\omega} \int_\omega \frac{|v(x)-v(y)|^2}{|x-y|^{2s+2}}\,dy\,dx
  \\
  &= 2\int_\omega |v(y)|^2\int_{\Omega\setminus\omega}\frac{1}{|x-y|^{2s+2}}\,dx\,dy\\
  &\lesssim \int_\omega |v(y)|^2\,dy=\norm{v}{L^2({\rm supp}(v))}^2,
 \end{align*}
 where we used that in polar coordinates centered at $y$, the inner integral reads
 \begin{align*}
  \int_{\Omega\setminus\omega}\frac{1}{|x-y|^{2s+2}}\,dx=\int_{\Omega\setminus\omega} r^{-2s}\,dr \lesssim 1.
 \end{align*}
 This, together with norm equivalence discussed in Definition~\ref{def:sobslob} concludes the proof.
\end{proof}

\begin{lemma}\label{lem:scaling}
 Let $k\leq \ell\in\N$, let $v\in \PP^2(\widehat\TT_k)$ with ${\rm supp}(v)\subseteq \bigcup \omega^n(T,\widehat\TT_k)$ for some $T\in\widehat\TT_k$. Additionally, let $w\in B^1_\ell\cup B^0_\ell$.
 Then, there holds for $0\leq s< 3/2$ if $v\in H^s({\rm supp}(v)\cup{\rm supp}(w))$ that
 \begin{align*}
  \norm{v}{H^s({\rm supp}(w))}\leq C_{\rm sc}C_{\rm mesh}^{k-\ell}\norm{v}{H^s({\rm supp}(v))}.
 \end{align*}
The constant $C_{\rm sc}>0$ depends only on $\TT_0$ and $n\in\N$.
\end{lemma}
\begin{proof}
 First, we prove the statement for $s=0$.
 To that end, consider the affine transformations $\phi_v,\phi_w\colon \R^2\to\R^2$ with $\phi_v(\omega_v)= {\rm supp}(v)$ and $\phi_w(\omega_w)= {\rm supp}(w)$, where
 $\omega_v$ and $\omega_w$ have unit area and belong to a finite family of shapes depending only on $\TT_0$ and $n\in\N$. This is possible since, by assumption, the supports
 of $v$ and $w$ are contained within a generalized patch and newest-vertex bisection produces only finitely many patch shapes. We obtain
 \begin{align*}
 \norm{v}{L^2({\rm supp}(w))}^2&=|{\rm supp}(w)|\norm{v\circ \phi_w}{L^2(\omega_w)}^2\leq |{\rm supp}(w)|\norm{v\circ \phi_w}{L^\infty(\omega_w)}^2\\
 &\leq 
 |{\rm supp}(w)|\norm{v\circ \phi_v}{L^\infty(\omega_v)}^2,
 \end{align*}
where we used the fact that $\phi_v(\omega_v)={\rm supp}(v)$ in the last estimate. Since $v\circ \phi_v$ belongs to a finite dimensional space which depends only on $\omega_v$, we further conclude
\begin{align*}
  |{\rm supp}(w)|\norm{v\circ \phi_v}{L^\infty(\omega_v)}^2&\lesssim
 |{\rm supp}(w)|\norm{v\circ \phi_v}{L^2(\omega_v)}^2\simeq \frac{|{\rm supp}(w)|}{|{\rm supp}(v)|}\norm{v}{L^2({\rm supp}(v))}^2.
\end{align*}
Both estimates and the support size estimates in Theorems~\ref{thm:stableB1}\&\ref{thm:stableB0} conclude the statement for $s=0$.
By applying the case $s=0$ to $\nabla v$ and $w$, we immediately prove the case $s=1$.
To get the intermediate cases, we employ the interpolation lemma for the operators $T^s_w\colon \PP^2(\omega^n(T,\widehat\TT_k))\cap H^s(\Omega)\to \PP^2(\omega^n(T,\widehat\TT_k)\cap{\rm supp}(w))\cap H^s(\Omega)$.
The above shows that $T_w^1$ and $T_w^0$ are uniformly bounded for all $w\in B_\ell^1\cup B_\ell^0$, with operator norms $C_{\rm sc}C_{\rm mesh}^{\ell-k}$. The interpolation lemma
implies boundedness for all $T_w^s$, $0\leq s\leq 1$, where the operators norms are independent of $T$ or $w$.
The cases $1<s<3/2$ follow from the same argument applied to $\nabla v$.
\end{proof}

\begin{lemma}\label{lem:Ascaling}
Let $M_{ij}:=\dual{\nabla v_i}{\nabla v_j}_\Omega$ for all $i,j\in\N$ with $v_i,v_j\in B^{1}$.
Given $\eps>0$, there exists $M^\eps\in\R^{\N\times \N}$ and a constant $C_M>0$ such that
\begin{align*}
 \norm{M^\eps-M}{2}\leq \eps.
\end{align*}
as well as
\begin{align}\label{eq:Abanded}
 \Big(|L(v_i)-L(v_j)|> C_M \text{ or } d_1(v_i,v_j)> C_M\Big)\quad\implies\quad M^\eps_{ij}=0.
\end{align}
\end{lemma}
\begin{proof}
Assume $k:=L(v_i)\leq \ell:=L(v_j)$. There holds
\begin{align*}
|\dual{\nabla v_i}{\nabla v_j}_\Omega| &\leq \norm{\nabla v_i}{H^{1/4}({\rm supp}(v_j))}\norm{\nabla v_j}{\widetilde H^{-1/4}({\rm supp}(v_j))}.
\end{align*}
Lemma~\ref{lem:scaling} proves 
$ \norm{\nabla v_i}{H^{1/4}({\rm supp}(v_j))}\lesssim \norm{v_i}{H^{5/4}({\rm supp}(v_j))}\lesssim C_{\rm mesh}^{-|k-\ell|}\norm{v_i}{H^{5/4}({\rm supp}(v_i))}$ as well as
\begin{align*}
 \norm{\nabla v_j}{\widetilde H^{-1}({\rm supp}(v_j))}&=\sup_{w\in H^{1}({\rm supp}(v_j))^2}\frac{\dual{\nabla v_j}{w}_{{\rm supp}(v_j)}}{\norm{w}{H^{1}({\rm supp}(v_j))}}\\
 &=\sup_{w\in H^{1}({\rm supp}(v_j))^2}\frac{-\dual{v_j}{{\rm div}w}_{{\rm supp}(v_j)}}{\norm{w}{H^{1}({\rm supp}(v_j))}}\\
 &\leq \norm{v_j}{L^2({\rm supp}(v_j))}.
\end{align*}
Together with the obvious estimate $\norm{\nabla v_j}{L^2({\rm supp}(v_j))}\leq \norm{v_j}{H^{1}({\rm supp}(v_j))}$, interpolation concludes $\norm{\nabla v_j}{\widetilde H^{-1/4}({\rm supp}(v_j))}\leq
\norm{v_j}{H^{3/4}({\rm supp}(v_j))}$. The above, together with~\eqref{eq:scaling1} shows
\begin{align}\label{eq:Ascaling}
 |\dual{\nabla v_i}{\nabla v_j}_\Omega|\lesssim  C_{\rm mesh}^{-5/4|k-\ell|}.
\end{align}
Symmetry of the problem shows the above also for $\ell\leq k$ and hence for all $i,j\in\N$.
We restrict the index set by
\begin{align*}
  I:=\set{(i,j)\in \N^2}{|L(v_i)-L(v_j)|\leq r}
 \end{align*}
 and define $M^\eps_{ij}:=M_{ij}$ for all $(i,j)\in I$ and zero elsewhere. 
 Note that $\#\set{v_j\in B_k^1}{M_{ij}\neq 0}\lesssim C_{\rm mesh}^{2(k-\min\{\ell,k\})}$ and $\#\set{v_i\in B^1_\ell}{M_{ij}\neq 0}\lesssim C_{\rm mesh}^{2(\ell-\min\{k,\ell\})}$.
 Estimate~\eqref{eq:Ascaling} and~\cite[Lemma~8.3]{fembemopt} with $q=C_{\rm mesh}^{-2}$ show $\norm{M-M^\eps}{2}\lesssim C_{\rm mesh}^{-r/4}$.
 The implication~\eqref{eq:Abanded} follows by definition of $I$.
 Thus, we conclude the proof by choosing $r\in\N$ sufficiently large.
\end{proof}

\begin{lemma}\label{lem:Bscaling}
Let $M_{ij}:=\dual{{\rm div}(v_i,0)}{w_j}_\Omega$ or $M_{ij}:=\dual{{\rm div}(0,v_i)}{w_j}_\Omega$ for all $i,j\in\N$ with $v_i\in B^{1}$ and $w_j\in B^0$.
Given $\eps>0$, there exists $M^\eps\in\R^{\N\times \N}$ and a constant $C_M>0$ such that
\begin{align*}
 \norm{M^\eps-M}{2}\leq \eps.
\end{align*}
as well as
\begin{align}\label{eq:Bbanded}
 \Big(|L(v_i)-L(w_j)|> C_M \text{ or } d_1(v_i,w_j)> C_M\Big)\quad\implies\quad M^\eps_{ij}=0.
\end{align}
\end{lemma}
\begin{proof}
Assume $k:=L(v_i)\geq \ell:=L(w_j)$.
Then, there holds with~\eqref{eq:scaling1}
\begin{align*}
 |M_{ij}|&=|\dual{(v_i,0)}{\nabla w_j}_\Omega |\leq \norm{v_i}{L^2({\rm supp}(v_i))}\norm{w_j}{H^1({\rm supp}(v_i))}\\
 &\lesssim C_{\rm mesh}^{-k}\norm{w_j}{H^1({\rm supp}(v_i))}.
\end{align*}
From this, Lemma~\ref{lem:scaling} together with~\eqref{eq:scaling0} conclude
\begin{align*}
  \norm{w_j}{H^1({\rm supp}(v_i))}
 &\lesssim C_{\rm mesh}^{-|k-\ell|}\norm{w_j}{H^{1}({\rm supp}(w_j))}\simeq  C_{\rm mesh}^{-|k-\ell|+ \ell}.
\end{align*}
Altogether, we obtain $|M_{ij}|\lesssim C_{\rm mesh}^{-2|k-\ell|}$.
For $k<\ell$, we have with Lemma~\ref{lem:scaling} as well as~\eqref{eq:scaling1}\&\eqref{eq:scaling0} for some $0<s<1/2$
\begin{align*}
 |M_{ij}|&\leq \norm{v_i}{H^{1+s}({\rm supp}(w_j))}\norm{w_j}{H^{-s}({\rm supp}(w_j))}\\
 &\lesssim C_{\rm mesh}^{-s\ell}\norm{v_i}{H^{1+s}({\rm supp}(w_j))}\\
 &\lesssim 
 C_{\rm mesh}^{-s\ell-|\ell-k|}\norm{v_i}{H^{1+s}({\rm supp}(v_i))}
 \lesssim C_{\rm mesh}^{-(1+s)|\ell-k|}.
\end{align*}
We restrict the index set by
\begin{align*}
  I:=\set{(i,j)\in \N^2}{|L(v_i)-L(v_j)|\leq r}
 \end{align*}
 and define $M^\eps_{ij}:=M_{ij}$ for all $(i,j)\in I$ and zero elsewhere. 
 Note that $\#\set{v_j\in B_k^1}{M_{ij}\neq 0}\lesssim C_{\rm mesh}^{2(k-\min\{\ell,k\})}$ and $\#\set{v_i\in B^1_\ell}{M_{ij}\neq 0}\lesssim C_{\rm mesh}^{2(\ell-\min\{k,\ell\})}$.
 The above estimates  and~\cite[Lemma~8.3]{fembemopt} with $q=C_{\rm mesh}^{-2}$ show $\norm{M-M^\eps}{2}\lesssim C_{\rm mesh}^{-s r}$.
 The implication~\eqref{eq:Abanded} follows by definition of $I$.
 Thus, we conclude the proof by choosing $r\in\N$ sufficiently large.
\end{proof}

\section{A modified Scott-Zhang projection}\label{section:szb}
This section introduces the operators $S_\ell^1$ and $S_\ell^2$ used for constructing the Riesz basis in Theorems~\ref{thm:stableB1}\&\ref{thm:stableB0}. First, we construct some slightly modified Scott-Zhang operator in Lemma~\ref{lem:sz1} which is stable on $L^2(\Omega)$ and
satisfies some integral mean property. This is important as we use that operator on the pressure space $L^2_\star(\Omega)$ which contains functions with zero integral mean. In a second step in Theorems~\ref{thm:szb1}\&\ref{thm:szb2}, we use the Scott-Zhang operators to construct
the operators $S^1_\ell$ and $S^2_\ell$ which inherit all the nice properties from the Scott-Zhang operators but additionally
are nested in the sense that $S^i_\ell S_k^i = S^i_\ell$ for all $i\in\{1,2\}$ and all $\ell\leq k\in\N$. 

\begin{figure}
 \includegraphics[width=0.5\textwidth]{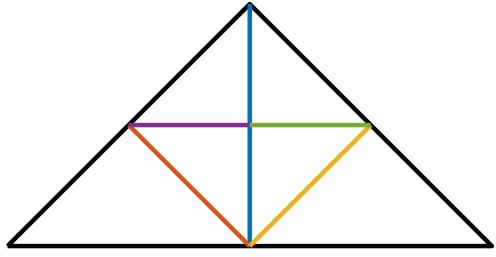}
 \centering
 \caption{The refinement pattern of {\rm bisec5} refinement.}\label{fig:bisec5}
\end{figure}

In the following, we define a particular basis of $\SS^{1}(\TT)$ with a certain moment condition.
\begin{definition}\label{def:sz1}
	Given a triangulation $\TT_c\in\T$ and it's bisec5 refinement $\TT\in{\rm bisec5}(\TT_c,\TT_c)$ (see Figure~\ref{fig:bisec5}) define for each
	 $T\in\TT_c$ the unique interior node $z_T\in\NN(\TT)$ with $z_T\in T\setminus\partial T$.
	 Define $\NN(\TT,\TT_c):=\set{z_T}{T\in\TT_c}$.
	For $z\in\NN(\TT)\setminus\NN(\TT,\TT_c)$ and the corresponding hat-function $v_z\in\SS^1_B(\TT)$, define
	\begin{align*}
	v_{z,\star}:=v_{z}+\sum_{\stacktwo{T\in\TT_c}{ v_z|_T\neq 0}}\alpha_T v_{z_T}
	\end{align*}
	where the $\alpha_T\in\R$ are
	chosen such that $\int_T v_{z,\star}\,dx =0$ for all $T\in\TT_c$.
	Define the basis
	\begin{align*}
	\SS^1_{B}(\TT,\TT_c):=	 \set{v_{z,\star}\in\SS^1_B(\TT)}{z\in\NN(\TT)\setminus\NN(\TT,\TT_c)}\cup\set{v_{z_T}\in\SS^1_B(\TT)}{T\in\TT_c}
	 \end{align*}
	 which satisfies $\SS^1(\TT)={\rm span}(\SS^1_B(\TT,\TT_c))$.
	For each $v\in\SS^1_B(\TT,\TT_c)$ define some $T_v\in\TT_c$ with $v|_{T_v}\neq 0$.
	Note that for all $T\in\TT_c$, the set $\set{v|_T}{v\in\SS^1_B(\TT,\TT_c)}\setminus\{0\}$ is a basis of $\SS^1(\TT|_T)$.
	This allows us to define the dual basis functions $v_T^\prime$ of $v|_T$ for all $v\in\SS^1_B(\TT,\TT_c)$ with $v|_T\neq 0$.
	Finally, we define $v^\prime:=v^\prime_{T_v}$ for all $v\in\SS^1_B(\TT,\TT_c)$.
	We define the Scott-Zhang projection as
	\begin{align*}
	 J_{\TT,\TT_c}^1 w:= \sum_{v\in\SS^1_B(\TT,\TT_c)} v\dual{w}{v^\prime}_{T_v}.
	\end{align*}
	\end{definition}

\begin{lemma}\label{lem:sz1}
	The Scott-Zhang operator $J_{\TT,\TT_c}^1\colon L^2(\Omega)\to \SS^1(\TT)$ defined in Definition~\ref{def:sz1} is a projection which satisfies for all $0\leq s\leq  1$
	and all $v\in H^s(\Omega)$
	\begin{align}
		\norm{J_{\TT,\TT_c}^1 v}{H^s(T)}&\leq C_{\rm sz} \norm{v}{H^s(\cup\omega(T,\TT_c))},\label{eq:szhs}\\
		\norm{J_{\TT,\TT_c}^1 v}{H^s(\Omega)}&\leq C_{\rm sz} \norm{v}{H^s(\Omega)},\label{eq:szOhs}
	\end{align}
	as well as  for all $0\leq r\leq s$
	\begin{align}
		\norm{(1-J_{\TT,\TT_c}^1) v}{H^r(T)}&\leq C_{\rm sz} {\rm diam}(T)^{s-r} |v|_{H^s(\cup\omega(T,\TT_c))},\label{eq:sza}\\
		\norm{(1-J_{\TT,\TT_c}^1) v}{H^r(\Omega)}&\leq C_{\rm sz} \norm{h_\TT^{1-r}\nabla v}{L^2(\Omega)}.\label{eq:szaO}
	\end{align}
	For $0<s<1/2$ and $v\in L^2(\Omega)$, $w\in \widetilde H^{-s}(\Omega)$, there holds
	\begin{align}
		\norm{J_{\TT,\TT_c}^1 w}{\widetilde H^{-s}(\Omega)}&\leq C_{\rm sz} \norm{w}{\widetilde H^{-s}(\Omega)},\label{eq:szOmhs}\\
		\norm{(1-J_{\TT,\TT_c}^1) v}{\widetilde H^{-s}(\Omega)}&\leq C_{\rm sz} \norm{h_\TT^{s}v}{L^2(\Omega)}.\label{eq:szmaO}
	\end{align}
	There holds $\int_T (1-J_{\TT,\TT_c})v\,dx =0$ for all $T\in\TT_c$.
	The constant $C_{\rm sz}>0$ depends only on the shape regularity of $\TT$, the fact that $\TT,\TT_c$ are generated from $\TT_0$ by newest vertex bisection and on a lower bound on $s>-1/2$. 
	The function $(J_{\TT,\TT_c}^1 v)|_T$ depends only on $v|_{\cup\omega(T,\TT_c)}$. 
\end{lemma}

\begin{proof}
	The projection property follows as usual from the definition of the dual basis functions $v^\prime$.
	For $T\in\TT_c$, $(J_\TT v)|_T$ depends only on $v|_{T^\prime}$ for all $T^\prime\in\TT_c$ with $T,T^\prime\subseteq {\rm supp}(w)$
	for some $w\in\SS^1_B(\TT,\TT_c)$. This implies $T^\prime\in\omega(T,\TT_c)$.
	Thus, $(J_\TT v)|_T$ depends only on $v|_{\omega(T,\TT_c)}$.
	
	The estimates~\eqref{eq:szhs}--\eqref{eq:szaO} follow analogously to the proof of~\cite[Section~3.2]{energy}. It remains to prove the estimates for $-s$.
	To that end, note that the dual function $v_{z_T}^\prime$ (where $z_T$ was defined as the unique interior node of the element $T\in \TT_c$) is a constant on $T$. This follows from the fact that 
	there holds $\dual{1}{v}_T=0$ for all $v\in\SS^1_B(\TT,\TT_c)\setminus\set{v_{z_T}}{T\in\TT_c}$ and $\dual{1}{v_{z_T}}_T>0$. Uniqueness of the dual function
	implies that $v_{z_T}^\prime$ is some multiple of $1$ and hence constant.
	This and the projection property imply for all $T\in\TT_c$, all $v\in L^2(\Omega)$, and $w:=v+\sum_{T\in\TT_c}\beta_Tv_{z_T}$
	\begin{align*}
		(1-J_{\TT,\TT_c})v&=	(1-J_{\TT,\TT_c})w\\
		&=w-\Big(\sum_{T\in\TT_c} v_{z_T}\dual{w}{v_{z_T}^\prime}_{T}
		+\sum_{v\in\SS^1_B(\TT,\TT_c)\setminus\set{v_{z_T}}{T\in\TT_c}} v\dual{w}{v^\prime}_T\Big).
	\end{align*}
	Choosing the $\beta_T$ such that $\int_T w\,dx=0$ for all $T\in\TT_c$, we obtain $\dual{w}{v_{z_T}^\prime}_T=0$ and hence
	\begin{align*}
		\int_T(1-J_{\TT,\TT_c})v\,dx=0\quad\text{for all }T\in\TT_c.
	\end{align*}
Moreover, we have for the $L^2$-orthogonal projection $\Pi^1\colon L^2(\Omega)\to \PP^1(\TT)$ the indentity $J_{\TT,\TT_c}^1 v= 
J_{\TT,\TT_c}^1 \Pi^1 v$. For $0<s<1/2$, this together with a standard Poincar\'e estimate and an inverse estimate imply
\begin{align*}
	\norm{(\Pi^1-J_{\TT,\TT_c}^1) v}{\widetilde H^{-s}(\Omega)}&=
	\norm{(1-J_{\TT,\TT_c}^1) \Pi^1 v}{\widetilde H^{-s}(\Omega)}\lesssim \norm{h_{\TT_c}^s \Pi^1 v}{L^2(\Omega)}
	\\
	&\lesssim \norm{h_{\TT_c}/h_\TT}{L^\infty(\Omega)}\norm{\Pi^1 v}{\widetilde H^{-s}(\Omega)}\lesssim 
	\norm{v}{\widetilde H^{-s}(\Omega)},
\end{align*}
where we used the continuity $\Pi^1\colon \widetilde H^{-s}(\Omega)\to \widetilde H^{-s}(\Omega)$.
Together with the approximation properties of $\Pi^1$, this shows
\begin{align*}
	\norm{J_{\TT,\TT_c}^1 v}{\widetilde H^{-s}(\Omega)}\lesssim \norm{\Pi^1 v}{\widetilde H^{-s}(\Omega)}+\norm{(\Pi^1-J_{\TT,\TT_c}^1) v}{\widetilde H^{-s}(\Omega)}\lesssim
	\norm{v}{\widetilde H^{-s}(\Omega)}
\end{align*}
as well as
\begin{align*}
	\norm{(1-J_{\TT,\TT_c}^1) v}{\widetilde H^{-s}(\Omega)}\lesssim \norm{(1-\Pi^1) v}{\widetilde H^{-s}(\Omega)}+\norm{(\Pi^1-J_{\TT,\TT_c}^1) v}{\widetilde H^{-s}(\Omega)}\lesssim
	\norm{h_{\TT_c}^s v}{L^2(\Omega)}
\end{align*}
and concludes~\eqref{eq:szOmhs}--\eqref{eq:szmaO}.
\end{proof}

For the next theorem, note that for all $\ell\geq 1$, there exists a coarsening $\overline\TT_\ell$ of $\widehat\TT_\ell$ such that 
$\widehat\TT_\ell={\rm bisec5}(\overline\TT_\ell,\overline\TT_\ell)$.

\begin{theorem}\label{thm:szb1}
	Recall $\widehat\TT_\ell$ and $C_{\rm mesh}$ from Definition~\ref{def:widehat}.
	With $J_\ell^1:= J_{\widehat\TT_\ell,\overline\TT_{\ell}}^1$, define
	\begin{align*}
		S_\ell^1 v := \lim_{N\to\infty} (J_\ell^1 J_{\ell+1}^1\ldots J_{\ell+N}^1)v\in S^{1}_\star(\widehat \TT_\ell)
	\end{align*}
	for all $v\in L^2(\Omega)$ and all $\ell\geq 1$.
	Then, the operator $S^1_\ell\colon L^2(\Omega)\to \SS^{1}(\widehat \TT_\ell)$ is well-defined and satisfies 
	\begin{align}
		\norm{(1-S_\ell^1) v}{\widetilde H^{-s}(\Omega)}&\leq C_{\rm S} C_{\rm mesh}^{-s\ell }\norm{v}{L^2(\Omega)}\quad\text{for all }v\in L^2(\Omega) \text{ and }0<s<1/2,\label{eq:Sapprox1}\\
		\norm{(1-S_\ell^1) v}{\widetilde H^{-s}(\Omega)}&\leq C_{\rm S} C_{\rm mesh}^{-(s-r)\ell }\norm{v}{\widetilde H^{-r}(\Omega)}\quad\text{for all }v\in \widetilde H^{-r}(\Omega),
		0<r\leq s-1/4,\label{eq:Sapprox15}\\
		\norm{(1-S_\ell^1) v}{L^2(\Omega)}&\leq C_{\rm S} C_{\rm mesh}^{-\ell/4}\norm{v}{H^{1/4}(\Omega)}\quad\text{for all }v\in H^{1/4}(\Omega),\label{eq:Sapprox2}\\
		\norm{S_\ell^1 v}{\widetilde H^{-1/4}(\Omega)}&\leq C_{\rm S} \norm{ v}{\widetilde H^{-1/4}(\Omega)}\quad\text{for all }v\in \widetilde H^{-1/4}(\Omega), \label{eq:Scont1}\\
		\norm{S_\ell^1 v}{H^{s}(\Omega)}&\leq C_{\rm S} \norm{ v}{ H^{s}(\Omega)}\quad\text{for all }v\in H^{s}(\Omega)\text{ and }s\in\{0,1/4\}. \label{eq:Scont2}
	\end{align}
	 Moreover, there holds $S_\ell^1 S_k^1=S_{\min\{\ell,k\}}^1$ for all $\ell,k\in\N_0$ as well as
	 \begin{align}\label{eq:Sid}
	  S_\ell^1 v = J_\ell^1 J_{\ell+1}^1\cdots J_{\ell+k}^1 v\quad\text{for all }v\in \SS^{1}(\widehat\TT_{\ell+k})
	 \end{align}
	for all $\ell,k\in\N$.
	The constant $C_{\rm S}>0$ depends only on $C_{\rm sz}$ and $C_{\rm mesh}$, $\TT_0$, and $s$.
\end{theorem}
\begin{proof}
For brevity of presentation, we also write $H^{-s}(\Omega)$ for $\widetilde H^{-s}(\Omega)$ in this proof.
	From Lemma~\ref{lem:sz1}, we obtain for $-1/2< r\leq 1$ and $s=1$
	\begin{align}\label{eq:szaOO}
		\norm{(1-J^1_\ell)v}{H^r(\Omega)}\lesssim C_{\rm mesh}^{1-r}\norm{v}{H^{s}(\Omega)}\quad\text{for all }v\in H^s(\Omega).
	\end{align}
	Moreover, from~\eqref{eq:szOhs}, we even get $\norm{(1-J^1_\ell)v}{H^r(\Omega)}\lesssim \norm{v}{H^{r}(\Omega)}$ for all $0\leq r\leq 1$.
	Interpolation arguments prove that~\eqref{eq:szaOO} holds for all $s$ with $r\leq s \leq 1$.
	Let $u\in H^{\mu}(\Omega)$. Since $J^1_\ell\colon H^\nu(\Omega)\to H^\nu(\Omega)$ is continuous for $-1/2< \nu\leq 1$ (see Lemma~\ref{lem:sz1}), 
	there holds for $-1/2< \nu\leq 1$ and $\nu+1/4\leq \mu \leq 1$ with $\mu\geq 0$ and by use of~\eqref{eq:szaOO} that
	\begin{align}\label{eq:sz1}
		\begin{split}
			\norm{(1-&(J^1_\ell J^1_{\ell+1}\ldots J^1_{\ell+N}))u}{H^{\nu}(\Omega)}\\
			&\leq \norm{(1-J^1_{\ell})u}{H^{\nu}(\Omega)}+\sum_{k=0}^{N-1}\norm{( (J^1_{\ell}\cdots J^1_{\ell+k})-(J^1_{\ell}\cdots J^1_{\ell+k+1}))u}{H^{\nu}(\Omega)}\\
			&\leq \norm{(1-J^1_{\ell})u}{H^{\nu}(\Omega)}+\sum_{k=0}^{N-1}\norm{(J^1_{\ell}\cdots J^1_{\ell+k})}{H^{\nu}(\Omega)\to H^{\nu}(\Omega)}\norm{(1-J^1_{\ell+k+1})u}{H^{\nu}(\Omega)}\\
			&\lesssim \sum_{k=0}^{N}C_{\rm sz}^{k}C_{\rm mesh}^{-(\mu-\nu)(\ell+k)}\norm{u}{H^\mu(\Omega)}\lesssim C_{\rm mesh}^{-(\mu-\nu)\ell}\norm{u}{H^\mu(\Omega)},
		\end{split}
	\end{align}
	where we used $C_{\rm mesh}^{(\mu-\nu)}\geq C_{\rm mesh}^{1/4}>C_{\rm sz}$ from Definition~\ref{def:widehat}.
	Hence,~\eqref{eq:sz1} for $ (\nu_1,\mu_1)= (0,1/4)$ and $(\nu_2,\mu_2)=(1/4,1/2)$ implies for $M\leq N$
	\begin{align*}
		\norm{((J^1_\ell J^1_{\ell+1}\ldots J^1_{\ell+M})&-(J^1_\ell J^1_{\ell+1}\ldots J^1_{\ell+N}))u}{L^2(\Omega)}\\
		&=\norm{(J^1_\ell\ldots J^1_{\ell+M})((1-(J^1_{\ell+M+1} J^1_{\ell+1}\ldots J^1_{\ell+N}))u}{L^2(\Omega)}\\
		&\leq \norm{((1-(J^1_{\ell+M+1} J^1_{\ell+1}\ldots J^1_{\ell+N}))u}{L^2(\Omega)}\\
		&\qquad+\norm{(1-(J^1_\ell\ldots J^1_{\ell+M}))((1-(J^1_{\ell+M+1} J^1_{\ell+1}\ldots J^1_{\ell+N}))u}{L^2(\Omega)}\\
		&\lesssim \norm{((1-(J^1_{\ell+M+1} J^1_{\ell+1}\ldots J^1_{\ell+N}))u}{H^{1/4}(\Omega)}\\
		&\lesssim C_{\rm mesh}^{-(\ell+M)/4}\norm{u}{H^{1/2}(\Omega)}.
	\end{align*}
	This shows, that $(J^1_\ell J^1_{\ell+1}\ldots J^1_{\ell+N})u$ is a Cauchy-sequence in $\SS^{1}(\widehat\TT_\ell)$ with respect to the $L^2(\Omega)$-norm as $N\to\infty$.
	Thus, for $u\in H^{1/2}(\Omega)$, the limit $S_\ell^1 u\in \SS^{1}(\widehat\TT_\ell)$ exists.  
	The estimates~\eqref{eq:Sapprox1}--\eqref{eq:Sapprox2} follow from~\eqref{eq:sz1}, the convergence $\lim_{N\to\infty}(J^1_\ell\ldots J^1_{\ell+N})v= S_\ell v$ in $L^2(\Omega)$,
	and density arguments. 
	
	To see~\eqref{eq:Scont1}--\eqref{eq:Scont2}, we apply inverse estimates for $s\in\{-1/4,0,1/4\}$ as well as Lemma~\ref{lem:sz1} and~\eqref{eq:Sapprox1}--\eqref{eq:Sapprox15} to see
	\begin{align*}
		\norm{S_\ell^1 u}{H^s(\Omega)}&\lesssim \norm{(S_\ell^1 -J^1_\ell) u}{H^s(\Omega)}+\norm{ u}{H^s(\Omega)}
		\\
		&\lesssim C_{\rm mesh}^{\ell(2/5+s)}\norm{(S_\ell^1 -J^1_\ell) u}{\widetilde H^{-2/5}(\Omega)}+\norm{ u}{H^s(\Omega)}\\
		&\lesssim C_{\rm mesh}^{\ell(2/5+s)}\big(\norm{(1-S_\ell^1) u}{\widetilde H^{-2/5}(\Omega)}+\norm{(1-J^1_\ell) u}{\widetilde H^{-2/5}(\Omega)}\big)+\norm{ u}{H^s(\Omega)}\\
		&\lesssim \norm{ u}{H^s(\Omega)},
	\end{align*}	
	The proof of $S_k S_\ell^1 u = S_k u$,~\eqref{eq:Sid}, and the projection property follows analogously to~\cite[Theorem~5.6]{fembemopt}.
\end{proof}

\begin{lemma}\label{lem:sz2}
	Define the Scott-Zhang operator $J_\TT^2\colon H^1(\Omega)\to \SS^2(\TT)$ from~\cite{scottzhang}. $J_\TT^2$ is a projection which satisfies for all $1/2< s<3/2$
	and all $v\in H^s(\Omega)$
	\begin{align}
		\norm{J_\TT^2 v}{H^s(T)}&\leq C_{\rm sz} \norm{v}{H^s(\cup\omega(T,\TT))},\label{eq:2szhs}\\
		\norm{J_\TT^2 v}{H^s(\Omega)}&\leq C_{\rm sz} \norm{v}{H^s(\Omega)},\label{eq:2szOhs}
	\end{align}
	as well as  for all $0\leq r\leq s$
	\begin{align}
		\norm{(1-J_\TT^2) v}{H^r(T)}&\leq C_{\rm sz} {\rm diam}(T)^{s-r} |v|_{H^s(\cup\omega(T,\TT))},\label{eq:2sza}\\
		\norm{(1-J_\TT^2) v}{H^r(\Omega)}&\leq C_{\rm sz} \norm{h_\TT^{1-r}\nabla v}{L^2(\Omega)}.\label{eq:2szaO}
	\end{align}
	The constant $C_{\rm sz}>0$ depends only on the shape regularity of $\TT$, the fact that $\TT$ is generated from $\TT_0$ by newest vertex bisection and on a lower bound on $s>1/2$. 
	The function $(J_\TT^2 v)|_T$ depends only on $v|_{\cup\omega(T,\TT)}$ and $v|_{\partial\Omega}=0$ implies $(J_\TT^2 v)|_{\partial\Omega}=0$. 
\end{lemma}

\begin{theorem}\label{thm:szb2}
	With $J_\ell^2:= J_{\widehat\TT_\ell}^2\colon H^1(\Omega)\to \SS^2(\widehat\TT_\ell)$ denoting the standard Scott-Zhang projection from Lemma~\ref{lem:sz2}, define
	\begin{align*}
		S_\ell^2 v := \lim_{N\to\infty} (J_\ell^2 J_{\ell+1}^2\ldots J_{\ell+N}^2)v\in S^{2}(\widehat \TT_\ell)
	\end{align*}
	for all $v\in H^1(\Omega)$.
	Then, the operator $S^2_\ell\colon H^1(\Omega)\to \SS^{1}(\widehat \TT_\ell)$ is well-defined and satisfies 
	\begin{align}
		\norm{(1-S_\ell^2) v}{ H^{1-s}(\Omega)}&\leq C_{\rm S} C_{\rm mesh}^{-s\ell }\norm{v}{H^1(\Omega)}\quad\text{for all }v\in H^1(\Omega) \text{ and }0<s<1/2,\label{eq:S2approx1}\\
		\norm{(1-S_\ell^2) v}{H^{1}(\Omega)}&\leq C_{\rm S} C_{\rm mesh}^{-\ell/4}\norm{v}{H^{5/4}(\Omega)}\quad\text{for all }v\in H^{1}(\Omega),\label{eq:S2approx2}\\
		\norm{S_\ell^2 v}{H^{s}(\Omega)}&\leq C_{\rm S} \norm{ v}{ H^{s}(\Omega)}\quad\text{for all }v\in H^{s}_\star(\Omega)\text{ and }s\in\{3/4,1,5/4\}. \label{eq:S2cont2}
	\end{align}
	 Moreover, there holds $S_\ell^2 S_k^2=S_{\min\{\ell,k\}}^2$ for all $\ell,k\in\N_0$ as well as
	 \begin{align}\label{eq:S2id}
	  S_\ell^2 v = J_\ell^2 J_{\ell+1}^2\cdots J_{\ell+k}^2 v\quad\text{for all }v\in \SS^{1}_\star(\widehat\TT_{\ell+k})
	 \end{align}
	for all $\ell,k\in\N$. 
	The constant $C_{\rm S}>0$ depends only on $C_{\rm sz}$ and $C_{\rm mesh}$, $\TT_0$, and $s$.
\end{theorem}
 \begin{proof}
  The proof follows analogously to~\cite[Theorem~5.6]{fembemopt} and is therefore omitted.
 \end{proof}

\bibliographystyle{siamplain}
\bibliography{literature}
\end{document}